\documentclass[10pt]{article}

\usepackage{amsmath}
\usepackage{amssymb}
\usepackage{graphicx}
\usepackage{epic,eepic}

\usepackage{color}

\usepackage{amstext,amsthm,amssymb,amsmath}
\usepackage{graphicx}
\usepackage{subfigure}
\usepackage{wrapfig}
\usepackage{fullpage}
\usepackage{color}
\usepackage{multirow}
\usepackage{tabulary}
\usepackage{booktabs}
\usepackage{enumerate}

\newcommand{\average}[1]{\{\!\!\!\{{#1}\}\!\!\!\}}
\newcommand{\jump}[1]{[\![{#1}]\!]}

\newtheorem{theorem}{Theorem}[section]
\newtheorem{lemma}[theorem]{Lemma}

\newtheorem{remark}[theorem]{Remark}

\hfuzz=\maxdimen
\title{An adaptive generalized multiscale discontinuous Galerkin method (GMsDGM) for high-contrast flow problems}

\author{Eric T. Chung\thanks{Department of Mathematics, The Chinese University of Hong Kong, Hong Kong SAR.
This research is partially supported by the Hong Kong RGC General Research Fund (Project number: 400411).},
Yalchin Efendiev\thanks{Department of Mathematics, Texas A\&M University, College Station, TX; Numerical Porous Media SRI Center, King Abdullah University of Science and Technology (KAUST), Thuwal 23955-6900, Kingdom of Saudi Arabia}
and Wing Tat Leung\thanks{Department of Mathematics, Texas A\&M University, College Station, TX.}
}

\begin{document}
\maketitle

\begin{abstract}

In this paper, we develop an adaptive Generalized Multiscale
Discontinuous Galerkin Method (GMsDGM)
for a class of high-contrast flow problems, and
derive a-priori and a-posteriori error estimates for the method.
Based on the a-posteriori error estimator,
we develop an adaptive enrichment algorithm for our GMsDGM
and prove its convergence.
The adaptive enrichment algorithm gives an automatic way
to enrich the approximation space in regions where the solution requires
more basis functions, which are shown to perform
well compared with a uniform enrichment.
We also discuss an approach that adaptively selects multiscale
basis functions by correlating the residual to multiscale basis functions
(cf. \cite{donoho_sparsity_review}).
The proposed error indicators are $L_2$-based and
can be inexpensively computed which makes our approach efficient.
Numerical results are presented that demonstrate the robustness
of the proposed error indicators.

\end{abstract}

\section{Introduction}
\label{sec:intro}

Model reduction techniques are often required for solving
challenging multiscale problems that have multiple scales and high contrast.
Many of these model
reduction techniques perform the discretization of the problem on a coarse
grid where coarse grid size is much larger than the fine-grid discretization.
The latter requires constructing reduced order models for the solution space
on a coarse grid. Some of these techniques involve upscaled models (e.g.,
\cite{dur91, weh02}) or multiscale methods
(e.g., \cite{Arbogast_two_scale_04, Chu_Hou_MathComp_10,ee03,
  egw10,eh09,ehg04, GhommemJCP2013,ReducedCon,MsDG,Wave,WaveGMsFEM}).

In this paper, we develop an adaptive Generalized Multiscale
Discontinuous Galerkin Method (GMsDGM)
for a class of high-contrast flow problems, and
derive a-priori and a-posteriori error estimates for the method.
We propose an adaptive enrichment algorithm for our GMsDGM
based on the a-posteriori error estimator
and prove its convergence.
The enrichment is done using inexpensive $L_2$-based error
indicators which allows adding
more basis functions in an automatic way.

The Generalized Multiscale Finite Element Method (GMsFEM), introduced in \cite{egh12},
is a generalization of the classical multiscale finite element method
(\cite{hw97}) by systematically enriching the coarse spaces and taking into account
small scale information and
complex input spaces.
While GMsFEM uses continuous Galerkin methods as coarse grid solvers,
the GMsDGM considered in this paper is based on
the interior penalty discontinuous Galerkin method as the coarse grid solver.
The discontinuous Galerkin formulation has several key advantages in multiscale finite element methods
(see \cite{eglmsMSDG}).
The basis functions for the GMsDGM
are totally decoupled across coarse element boundaries.
In addition, the GMsDGM is constructed following
the general framework on GMsFEM \cite{egh12}.
In particular, given a coarse grid partition of the computational domain,
a snapshot space is defined for each coarse element.
The snapshot space for a specific coarse element
contains all functions defined in the underlying fine grid.
A space reduction is then performed to obtain a much smaller space
by means of spectral decomposition.
Our analysis shows that we need two spectral problems
for the space reduction.
More precisely, the snapshot space is decomposed into two components,
which consists of harmonic extensions and functions that
are vanishing on coarse element boundaries.
Two separate spectral problems are used to compute reduced spaces.
A-priori error estimate is derived showing that the error
is inverse proportional to the first eigenvalue corresponding
to the first eigenfunction that is not used
in the construction of the reduced space.
We remark that similar results are obtained for GMsFEM \cite{egw10, eglp13}.

It is evident that different coarse elements need different
number of basis functions
in order to obtain accurate representation of the solution. For example, in less heterogeneous regions,
one needs fewer basis functions compared to the regions with more heterogeneities and high contrast.
It is therefore another objective of the paper to consider adaptive enrichment of basis functions.
We derive a-posteriori error estimate for the GMsDGM.
By using the error indicator, we construct an adaptive enrichment algorithm
and prove its convergence.
One important feature of our adaptive enrichment algorithm is
the ability to adaptively select basis functions in the space of harmonic extensions
and the space of functions vanishing on coarse element boundaries.
In addition, the error indicators are $L_2$-based, which can be computed efficiently.
This is an advantage over the $H^{-1}$-based
adaptive enrichment algorithm developed for GMsFEM \cite{Adaptive-GMsFEM}.
We also present a procedure for removing basis functions.
Our analysis is based on the idea in
\cite{BrennerScott,AdaptiveFEM,Adaptive-GMsFEM},
and do not consider the error due to the fine-grid
discretization of local problems and only study the errors due
to the enrichment.
In this regard, we assume that the error is largely
due to coarse-grid discretization.
The fine-grid discretization error can be considered in general
(e.g., as in \cite{abdul_yun, ohl12}) and
this will give an additional error estimator.
We remark that
there are many related activities in designing a-posteriori
error estimates \cite{Dorfler96,ohl12, abdul_yun, dinh13, nguyen13, tonn11}
 for global reduced models.
The main difference is that our error estimators are
based on special local eigenvalue problem and use the eigenstructure
of the offline space. We also discuss
an approach that adaptively selects multiscale
basis functions from the offline space by selecting
a basis with the most correlation to the local residual
(cf. \cite{donoho_sparsity_review}).

The rest of the paper is organized in the following way.
In the next section, we present the basic idea of GMsDGM.
The method is then detailed and analyzed
in Section \ref{cgdgmsfem}.
Then in Section \ref{sec:errorindicator},
we elaborate the adaptive algorithm and state the main convergence
results related to this algorithm. In Section \ref{sec:numerresults}, numerical
results are illustrated to test the performance of this adaptive algorithm.
The proofs of the main results are presented in Section \ref{sec:analysis}.

\section{Preliminaries}

\label{prelim} 
We will start this section with the problem settings and some notations.
Let $D$ be the computational domain
consisting of a medium modeled by the function $\kappa(x)$.
The high-contrast flow problem concerned in this paper is
\begin{equation}
-\mbox{div}\big(\kappa(x)\,\nabla u\big)=f\quad\text{in}\quad D,\label{eq:original}
\end{equation}
subject to the homogeneous Dirichlet boundary condition $u=g$ on
$\partial D$.
The main difficulty in numerically solving the above problem is that
$\kappa(x)$ is highly heterogeneous with many scales and
high contrast.
We assume that $\kappa \geq 1$.
In order to efficiently obtain an approximate solution to (\ref{eq:original}), we will need
the notion of fine and coarse grids.

Consider a given triangulation $\mathcal{T}^{H}$ of the domain $D$
with mesh size $H>0$. We call $\mathcal{T}^{H}$ the coarse
grid and $H$ the coarse mesh size. Elements of $\mathcal{T}^{H}$
are called coarse grid blocks and we use $N$ to denote the number of coarse grid blocks.
The set of all coarse grid edges is
denoted by $\mathcal{E}^{H}$.
We also introduce a finer triangulation $\mathcal{T}^{h}$ of the computational domain $D$,
obtained by
a conforming refinement of
the coarse grid $\mathcal{T}^{H}$. We call $\mathcal{T}^{h}$ the
fine grid and $h>0$ the fine mesh size. We remark that the use of
the conforming refinement is only to simplify the discussion of the
methodology and is not a restriction of the method.

Now we present the framework of GMsDGM. The methodology consists
of two main ingredients, namely, the construction of local basis functions
and the global coarse grid level coupling.
For the local basis functions,
a snapshot space $V^{i,\text{snap}}$ is
first constructed for each coarse grid block $K_{i} \in\mathcal{T}^H$.
The snapshot space contains a large library of basis functions, which can be used to obtain
a fine scale approximate solution to (\ref{eq:original}).
A spectral problem
is then solved in the snapshot space $V^{i,\text{snap}}$ and eigenfunctions corresponding
to dominant modes are used as the final basis functions. The resulting space
is called the local offline space $V^{i,\text{off}}$ for the $i$-th coarse
grid block $K_i$. The global offline space $V^{\text{off}}$ is then defined
as the linear span of all these $V^{i,\text{off}}$, for $i=1,2,\cdots,N$.
This global offline space $V^{\text{off}}$ will be used as the approximation space of
our discontinuous Galerkin method, which can be
formulated as: find $u_{H}^{\text{DG}}\in V^{\text{off}}$
such that
\begin{equation}
a_{\text{DG}}(u_{H}^{\text{DG}},v)=(f,v),\quad\forall v\in V^{\text{off}},\label{eq:ipdg}
\end{equation}
where the bilinear form $a_{\text{DG}}$ is defined as
\begin{equation}
a_{\text{DG}}(u,v)=a_{H}(u,v)-\sum_{E\in\mathcal{E}^{H}}\int_{E}\Big(\average{{\kappa}\nabla{u}\cdot{n}_{E}}\jump{v}+\average{{\kappa}\nabla{v}\cdot{n}_{E}}\jump{u}\Big)+\sum_{E\in\mathcal{E}^{H}}\frac{\gamma}{h}\int_{E}\overline{\kappa}\jump{u} \jump{v} \label{eq:bilinear-ipdg}
\end{equation}
with
\begin{equation}
a_{H}({u},{v})=\sum_{K\in\mathcal{T}_{H}}a_{H}^{K}(u,v),\quad a_{H}^{K}(u,v)=\int_{K}\kappa\nabla u\cdot\nabla v,
\end{equation}
where $\gamma>0$ is a penalty parameter, ${n}_{E}$ is a fixed unit
normal vector defined on the coarse edge $E \in \mathcal{E}^H$.
Note that, in (\ref{eq:bilinear-ipdg}),
the average and the jump operators are defined in the classical way.
Specifically, consider an interior coarse edge $E\in\mathcal{E}^{H}$
and let $K^{+}$ and $K^{-}$ be the two coarse grid blocks sharing
the edge $E$. For a piecewise smooth function $G$, we define
\[
\average{G}=\frac{1}{2}(G^{+}+G^{-}),\quad\quad\jump{G}=G^{+}-G^{-},\quad\quad\text{ on }\, E,
\]
where $G^{+}=G|_{K^{+}}$ and $G^{-}=G|_{K^{-}}$ and we assume that
the normal vector ${n}_{E}$ is pointing from $K^{+}$ to $K^{-}$.
Moreover, on the edge $E$, we define $\overline{\kappa} = (\kappa_{K^+}+\kappa_{K^-})/2$
where $\kappa_{K^{\pm}}$ is the maximum value of $\kappa$ over $K^{\pm}$.
For a coarse edge $E$ lying on the boundary $\partial D$, we define
\[
\average{G}=\jump{G}=G,\quad \text{ and }\quad \overline{\kappa} = \kappa_{K} \quad\quad\text{ on }\, E,
\]
where we always assume that ${n}_{E}$ is pointing outside of $D$.
We note that the DG coupling (\ref{eq:ipdg})
is the classical interior penalty discontinuous Galerkin (IPDG) method \cite{IPDGbook}
with our multiscale basis functions as the approximation space.


\section{GMsDGM for high-contrast flow problems}

\label{cgdgmsfem} 

In this section, we will give a detailed description of the method.
We will first give the construction of the snapshot space, and then
give the definitions of the local spectral problems for the construction of the offline space.
Furthermore, a priori estimate of the method will be derived.

Let $K_{i} \in\mathcal{T}^H$ be a given coarse grid block.
We will define two types of snapshot spaces.
The first type of local snapshot space
$V^{i,\text{snap}}_1$ for the coarse grid block $K_{i}$ is defined
as the linear span of all harmonic extensions. Specifically, given
a function $\delta_{k}$ defined on $\partial K_i$, we find ${\psi}_{k}^{i,\text{snap}} \in V_h(K_i)$ by
\begin{equation}
\begin{split}\int_{K_{i}}\kappa\nabla\psi_{k}^{i,\text{snap}} \cdot\nabla v & ={0},\quad\text{\ensuremath{\forall}}\; v\in V^0_{h}(K_{i}), \\
\psi_{k}^{i,\text{snap}} & =\delta_{k},\quad\text{ on }\;\partial K_{i},
\end{split}
\label{eq:snapproblem1}
\end{equation}
where $V_h(K_i)$ is the standard conforming piecewise linear finite element space with respect to the fine grid defined on $K_i$,
$V_h^0(K_i)$ is the subspace of $V_h(K_i)$ containing functions vanishing on $\partial K_i$
and $\delta_k$ is piecewise linear on $\partial K_i$ with respect to the fine grid
such that $\delta_k$ has the value one at the $k$-th fine grid node and value zero at all the remaining fine grid nodes.
The linear span of the above harmonic extensions is the local snapshot
space $V^{i,\text{snap}}_1$, namely
\[
V^{i,\text{snap}}_1=\text{span}\{{\psi}_{k}^{i,\text{snap}},\quad k=1,2,\cdots,M^{i,\text{snap}}\},
\]
where $M^{i,\text{snap}}$ is the number of basis functions in $V^{i,\text{snap}}_1$,
which is also equal to the number of fine grid nodes on $\partial K_i$.
The second type of local snapshot space $V^{i,\text{snap}}_2$ for the coarse grid
block $K_{i}$ is defined as $V^{i,\text{snap}}_2 = V_h^0(K_i)$.
It is easy to see that $V_h(K_i) = V^{i,\text{snap}}_1 + V^{i,\text{snap}}_2$,
namely the space $V_h(K_i)$ is decomposed as the sum of harmonic extensions
and functions vanishing on the boundary $\partial K_i$.
Moreover, the global snapshot space $V^{\text{snap}}_1$
is defined so that any $v\in V^{\text{snap}}_1$
if $v|_{K_i} \in V^{i,\text{snap}}_1$.
The space $V^{\text{snap}}_2$ is defined similarly.

We will perform dimension reductions on the above snapshot spaces
by the use of some carefully selected spectral problems. Based on our analysis to be presented
in this section, we define the spectral problem for $V^{i,\text{snap}}_1$
as finding eigenpairs $(\phi_k^{(i)}, \lambda_{1,k}^{(i)})$, $k=1,2,\cdots, M^{i,\text{snap}}$, such that
\begin{equation}
\int_{K_{i}}\kappa\nabla\phi_k^{(i)} \cdot\nabla v=\frac{\lambda_{1,k}^{(i)}}{H}\int_{\partial K_{i}}\widetilde{\kappa}{\phi_k^{(i)}}{v},\quad\quad\forall v\in V^{i,\text{snap}}_1,\label{eq:spec-dg}
\end{equation}
where $\widetilde{\kappa}$ is the maximum of $\overline{\kappa}$ over all coarse edges $E\in\partial K_i$.
Moreover, we assume that
\begin{equation*}
\lambda_{1,1}^{(i)} < \lambda_{1,2}^{(i)} < \cdots < \lambda_{1,M^{i,\text{snap}}}^{(i)}.
\end{equation*}
For the space $V^{i,\text{snap}}_2$, we define the spectral problem as finding eigenpairs $(\xi_k^{(i)}, \lambda_{2,k}^{(i)})$, $k=1,2,\cdots$, such that
\begin{equation}
\int_{K_{i}}\kappa\nabla\xi_k^{(i)}\cdot\nabla v=\frac{\lambda_{2,k}^{(i)}}{H^{2}}\int_{K_{i}}\kappa{\xi_k^{(i)}}{v}, \quad\quad\forall v\in V^{i,\text{snap}}_2,\label{eq:spec-dg-1}
\end{equation}
where we also assume that
\begin{equation*}
\lambda_{2,1}^{(i)} < \lambda_{2,2}^{(i)} < \cdots
\end{equation*}
In the spectral problems (\ref{eq:spec-dg}) and (\ref{eq:spec-dg-1}), we
will take respectively the first $l_{1,i}$ and $l_{2,i}$ eigenfunctions to form
the offline space for the coarse grid block $K_i$.
The local offline spaces are then defined as
\begin{align*}
V^{i,\text{off}}_1 & =\text{span}\{{\phi}_{l}^{(i)},\quad l=1,2,\cdots,l_{1,i}\}, \\
V^{i,\text{off}}_2 & =\text{span}\{{\xi}_{l}^{(i)},\quad l=1,2,\cdots,l_{2,i}\}.
\end{align*}
We define $V^{i,\text{off}} = V^{i,\text{off}}_1 + V^{i,\text{off}}_2$.
The global offline space $V_1^{\text{off}}$
is defined so that the restriction of any function $v\in V_1^{\text{off}}$
on the coarse grid block $K_i$ belongs to $V_1^{i,\text{off}}$.
The definition for $V_2^{\text{off}}$ is defined similarly.
In addition, we define $V^{\text{off}} = V^{\text{off}}_1 + V^{\text{off}}_2$.
This space is used as the approximation space in (\ref{eq:ipdg}).

Now we will analyze the method defined in (\ref{eq:ipdg}).
For any piecewise smooth function ${u}$, we define the DG-norm by
\[
\|{u}\|_{\text{DG}}^{2}=a_{H}({u},{u})+\sum_{E\in\mathcal{E}_{H}}\frac{\gamma}{h}\int_{E}\overline{\kappa}\jump{u}^{2}\; ds.
\]
Let $K$ be a coarse grid block and let ${n}_{\partial K}$ be the
unit outward normal vector on $\partial K$. We denote $V_{h}(\partial K)$
by the restriction of the conforming space $V_{h}(K)$ on $\partial K$.
The normal flux $\kappa \nabla u \cdot \,{n}_{\partial K}$ is understood as
an element in $V_{h}(\partial K)$ and is defined by
\begin{equation}
\int_{\partial K}({\kappa}\nabla{u}\cdot{n}_{\partial K})\cdot{v}=\int_{K} \kappa\nabla u\cdot\nabla\widehat{{v}},\quad{v}\in V^{h}(\partial K),\label{eq:flux}
\end{equation}
where $\widehat{{v}} \in V_h(K)$ is the harmonic extension of ${v}$ in $K$.
By the Cauchy-Schwarz inequality,
\[
\int_{\partial K}({\kappa}\nabla{u}\cdot{n}_{\partial K})\cdot{v}\leq a_{H}^{K}(u,u)^{\frac{1}{2}}\, a_{H}^{K}(\widehat{v},\widehat{v})^{\frac{1}{2}}.
\]
By an inverse inequality and the fact that $\widehat{v}$ is the harmonic
extension of $v$
\begin{equation}
a_{H}^{K}(\widehat{v},\widehat{v})\leq\kappa_{K}C_{\text{inv}}^{2}h^{-1}\int_{\partial K}|v|^{2},
\label{eq:maxeig}
\end{equation}
where we recall that $\kappa_{K}$ is the maximum of $\kappa$ over $K$ and $C_{\text{inv}}>0$ is
the constant from inverse inequality. Thus,
\[
\int_{\partial K}({\kappa}\nabla{u}\cdot{n}_{\partial K})\cdot{v}\leq\kappa_{K}^{\frac{1}{2}}C_{\text{inv}}h^{-\frac{1}{2}}\|v\|_{L^{2}(\partial K)}\, a_{H}^{K}(u,u)^{\frac{1}{2}}.
\]
This shows that
\begin{equation}
\int_{\partial K}|{\kappa}\nabla{u}\cdot{n}_{\partial K}|^{2}\leq\kappa_{K}C_{\text{inv}}^{2}h^{-1}a_{H}^{K}(u,u).
\label{eq:fluxbound}
\end{equation}

Our first step in the development of an a priori estimate is to establish the continuity
and the coercivity of the bilinear form (\ref{eq:bilinear-ipdg})
with respect to the DG-norm.

\begin{lemma} Assume that the penalty parameter $\gamma$ is chosen
so that $\gamma>C_{\text{\rm inv}}^{2}$. The bilinear form $a_{\text{DG}}$
defined in (\ref{eq:bilinear-ipdg}) is continuous and coercive, that
is,
\begin{eqnarray}
a_{\text{DG}}({u},{v}) & \leq & a_{1}\|{u}\|_{\text{DG}}\,\|{v}\|_{\text{DG}}, \label{lem:cont} \\
a_{\text{DG}}({u},{u}) & \geq & a_{0}\|{u}\|_{\text{DG}}^{2}, \label{lem:coer}
\end{eqnarray}
for all ${u},{v}$, where $a_{0}=1-C_{\text{\rm inv}}\gamma^{-\frac{1}{2}}>0$ and $a_{1}=1+C_{\text{\rm inv}}\gamma^{-\frac{1}{2}}$.
\end{lemma}
\textit{Proof}. By the definition of $a_{\text{DG}}$,
we have
\[
a_{\text{DG}}({u},{v})=a_{H}({u},{v})-\sum_{E\in\mathcal{E}^{H}}\int_{E}\Big(\average{{\kappa}\nabla{u}\cdot{n}_{E}}\jump{v}+\average{{\kappa}\nabla{v}\cdot{n}_{E}}\jump{u}\Big)+\sum_{E\in\mathcal{E}^{H}}\frac{\gamma}{h}\int_{E}\overline{\kappa}\jump{u} \jump{v}.
\]
Notice that
\[
a_{H}({u},{v})+\sum_{E\in\mathcal{E}^{H}}\frac{\gamma}{h}\int_{E}\overline{\kappa}\jump{u}\cdot\jump{v} \leq\|u\|_{\text{DG}}\,\|v\|_{\text{DG}}.
\]
For an interior coarse edge $E\in\mathcal{E}^{H}$, we let $K^{+},K^{-}\in\mathcal{T}^{H}$
be the two coarse grid blocks having the edge $E$. By the Cauchy-Schwarz
inequality, we have
\begin{equation}
\int_{E}\average{{\kappa} \nabla u\cdot{n}_{E}}\cdot\jump{v}\leq\Big(h\int_{E}\average{{\kappa}\nabla{u}\cdot{n}_{E}}^{2}(\overline{\kappa})^{-1}\Big)^{\frac{1}{2}}\Big(\frac{1}{h}\int_{E}\overline{\kappa}\jump{v}^{2}\Big)^{\frac{1}{2}}.\label{eq:cont1}
\end{equation}
Notice that
\begin{equation*}
 h\int_{E}\average{{\kappa}\nabla{u}\cdot{n}_{E}}^{2}(\overline{\kappa})^{-1}
\leq  h\Big(\int_{E}({\kappa}^{+}\nabla{u}^{+}\cdot{n}_{E})^{2}(\kappa_{K^+})^{-1}+\int_{E}({\kappa}^{-}\nabla{u}^{-}\cdot{n}_{E})^{2}(\kappa_{K^-})^{-1}\Big)
\end{equation*}
where ${u}^{\pm}={u}|_{K^{\pm}}$, $\kappa^{\pm}=\kappa|_{K^{\pm}}$.
So, summing the above over all $E$ and by (\ref{eq:fluxbound}), we have
\begin{equation*}
h\sum_{E\in\mathcal{E}^{H}}\int_{E}\average{{\kappa}\nabla{u}\cdot{n}_{E}}^{2}(\overline{\kappa})^{-1}  \leq h\sum_{K\in\mathcal{T}_{H}}\int_{\partial K}(\kappa \nabla{u}\cdot{n}_{\partial K})^{2} (\kappa_{K})^{-1}
  \leq  C_{\text{inv}}^{2}a_{H}(u,u).
\end{equation*}
Thus we have
\begin{equation}
\sum_{E\in\mathcal{E}^{H}}\int_{E}\average{{\kappa}\nabla{u}\cdot{n}_{E}}\jump{v}\leq C_{\text{inv}}a_{H}(u,u)^{\frac{1}{2}}\Big(\sum_{E\in\mathcal{E}^{H}}\frac{1}{h}\int_{E}\overline{\kappa}\jump{v}^{2}\; ds\Big)^{\frac{1}{2}}.
\label{eq:cont5}
\end{equation}
Similarly, we have
\[
\sum_{E\in\mathcal{E}^{H}}\int_{E}\average{{\kappa}\nabla{v}\cdot{n}_{E}}\jump{u} \leq C_{\text{inv}}a_{H}(v,v)^{\frac{1}{2}}\Big(\sum_{E\in\mathcal{E}^{H}}\frac{1}{h}\int_{E}\overline{\kappa}\jump{u}^{2}\; ds\Big)^{\frac{1}{2}}.
\]
Summing the above two inequalities, we have
\begin{equation}
\sum_{E\in\mathcal{E}^{H}}\int_{E}\Big(\average{{\kappa}\nabla{u}\cdot{n}_{E}}\jump{v}+\average{{\kappa}\nabla{v}\cdot{n}_{E}}\jump{u}\Big)\leq C_{\text{inv}}\gamma^{-\frac{1}{2}}\|u\|_{\text{DG}}\,\|v\|_{\text{DG}}.\label{eq:cont4}
\end{equation}
This proves the continuity (\ref{lem:cont}).

For the coercivity (\ref{lem:coer}), we have
\[
a_{\text{DG}}({u},{u})=\|u\|_{\text{DG}}^{2}-\sum_{E\in\mathcal{E}^{H}}\int_{E}\Big(\average{\kappa\nabla u\cdot{n}_{E}}\cdot\jump{u}+\average{\kappa\nabla u\cdot{n}_{E}}\cdot\jump{u}\Big).
\]
By (\ref{eq:cont4}), we have
\[
a_{\text{DG}}({u},{u})\geq(1-C_{\text{inv}}\gamma^{-\frac{1}{2}})\|u\|_{\text{DG}}^{2},
\]
which gives the desired result.

\begin{flushright}
$\square$
\par\end{flushright}

In the following,
we will prove an a priori estimate of the method (\ref{eq:ipdg}).
First, we let
\begin{equation*}
V^h_{\text{DG}} = \{ v \in L^2(D) \, : \, v|_{K} \in V^h(K) \}.
\end{equation*}
Let $u_{h}\in V_{\text{DG}}^{h}$ be the fine grid solution which
satisfies
\begin{equation}
a_{\text{DG}}(u_{h},v)=(f,v),\quad\forall v\in V_{\text{DG}}^{h}.\label{eq:ipdgfine}
\end{equation}
It is well-known that $u_{h}$ converges to the exact solution $u$
in the DG-norm as the fine mesh size $h\rightarrow0$. Next, we define
a projection $u_{1}\in V^{\text{snap}}_1$ of $u_{h}$ in the snapshot
space by the following construction. For each coarse grid block $K_i$,
the restriction of $u_{1}$ on $K_i$ is defined as the harmonic extension
of $u_{h}$, that is,
\begin{equation}
\begin{split}\int_{K_{i}}\kappa\nabla u_{1}\cdot\nabla v & ={0},\quad\quad\text{\ensuremath{\forall}}\; v\in V^0_{h}(K_{i})\\
{u}_{1} & =u_{h},\quad\text{ on }\;\partial K_i.
\end{split}
\label{eq:proj}
\end{equation}



The following theorem
gives an a priori estimate for the GMsDGM (\ref{eq:ipdg}).

\begin{theorem} Let $u_{h}\in V_{\text{DG}}^{h}$ be the fine grid
solution defined in (\ref{eq:ipdgfine}) and $u_{H}$ be the GMsDGM
solution defined in (\ref{eq:ipdg}). Then we have
\begin{equation*}
\begin{split}
\|u_{h}-u_{H}\|_{\text{DG}}^{2}\leq
C\Big( & \sum_{i=1}^{N}\frac{H}{ \widetilde{\kappa} \lambda_{1,l_{1,i}+1}^{{(i)}}}(1+\frac{\gamma H}{h\lambda_{1,l_{1,i}+1}^{{(i)}}})
\int_{\partial K_{i}} (\kappa\nabla u_{1}\cdot n_{\partial K})^{2}
+\sum_{K\in\mathcal{T}_{H}}\cfrac{H^{2}}{\lambda_{2,l_{2,i}+1}^{{(i)}}} \|f\|_{L^2(K)}^{2} \\
& +C^2_{\text{\rm inv}} \sum_{E\in\mathcal{E}^H} \frac{1}{h} \int_E \overline{\kappa} \jump{u_h}^2
\Big),
\end{split}
\end{equation*}
where $u_{1}$ is defined in (\ref{eq:proj}).
\end{theorem}
\textit{Proof}.
First, we write $u_{h}=u_{1}+u_{2}$ where $u_{2}=u-u_{1}.$
Notice that, on each coarse grid block $K_{i}$, the functions $u_{1}$
and $u_{2}$ can be represented by
\begin{equation}
u_{1}=\sum_{l=1}^{M_{i}}c_{l}\phi_{l}^{(i)}
\quad\text{and}\quad
u_{2}=\sum_{l \geq 1}d_{l}\xi_{l}^{(i)}
\label{eq:expansion}
\end{equation}
where $M_{i}=M^{i,\text{snap}}$ and we assume that the functions
$\phi_{l}^{(i)}$ and $\xi_l^{(i)}$ are normalized so that
\[
\int_{\partial K_{i}}\overline{\kappa}(\phi_{l}^{(i)})^{2}=1
\quad\text{and}\quad
\int_{K_i} \kappa (\xi_l^{(i)})^2 = 1.
\]
Notice that, the functions $u_1$ and $u_2$ belong to the snapshot spaces $V^{\text{snap}}_1$ and $V^{\text{snap}}_2$ respectively.
We will need two functions
$\widehat{u}_{1}$ and $\widehat{u}_{2}$, which belong to the offline spaces $V^{\text{off}}_1$ and $V^{\text{off}}_2$ respectively.
These functions are
defined by
\begin{equation*}
\widehat{u}_{1}=\sum_{l=1}^{l_{1.i}}c_{l}\phi_{l}^{(i)}
\quad\text{and}\quad
\widehat{u}_{2}=\sum_{l=1}^{l_{2.i}}d_{l}\xi_{l}^{(i)}
\quad\quad\text{on } K_i.
\end{equation*}
We remark that $\widehat{u}_{1}$ and $\widehat{u}_{2}$ are the truncation of $u_1$ and $u_2$
up to the eigenfunctions selected to form the offline space.

Next,
we will find an estimate of $\|u_{1}-\widehat{u}_{1}\|_{\text{DG}}$.
Let $K_i\in\mathcal{T}^H$ be a given coarse grid block. Recall that the spectral problem
to form $V^{i,\text{off}}_1$ is
\[
\int_{K_{i}}\kappa\nabla \phi^{(i)}_k \cdot\nabla v=\frac{\lambda_{1,k}^{(i)}}{H}\int_{\partial K_{i}}\widetilde{\kappa} {\phi^{(i)}_k}{v},\quad\quad\forall v\in V^{i,\text{snap}}_1.
\]
By the definition of the flux defined in (\ref{eq:flux}), the above spectral problem
can be represented as
\[
\int_{\partial K_i}(\kappa\nabla \phi^{(i)}_k \cdot n_{\partial K_i})v\; ds=\frac{\lambda_{1,k}^{(i)}}{H}\int_{\partial K_{i}}\widetilde{\kappa}{\phi^{(i)}_k}{v}.
\]
By the definition of the DG-norm, the error $\|u_{1}-\widehat{u}_{1}\|_{\text{DG}}$
can be estimated by
\[
\|\widehat{{u}}_{1}-{u}_{1}\|_{\text{DG}}^{2}\leq\sum_{i=1}^N
\Big(\int_{K_i}\kappa|\nabla(\widehat{{u}}_{1}-{u}_{1})|^{2}+\frac{\gamma}{h}\int_{\partial K_i}\widetilde{\kappa}(\widehat{{u}}_{1}-{u}_{1})^{2}\Big).
\]
Note that, by (\ref{eq:expansion}), we have
\[
\int_{K_{i}}\kappa|\nabla(\widehat{{u}}_{1}-{u}_{1})|^{2}=\sum_{l={l_{1,i}}+1}^{M_{i}}\frac{\lambda_{1,l}^{(i)}}{H}c_{l}^{2}\leq\frac{H}{\lambda_{1,l_{1,i}+1}^{{(i)}}}\sum_{l={l_{1,i}}+1}^{M_{i}}(\frac{\lambda_{1,l}^{{(i)}}}{H})^{2}c_{l}^{2}
\]
and
\[
\frac{1}{h}\int_{\partial K_{i}}\widetilde{\kappa}(\widehat{{u}}_{1}-{u}_{1}))^{2}=\frac{1}{h}\sum_{l={l_{1,i}}+1}^{M_{i}}c_{l}^{2}\leq\frac{H^{2}}{h(\lambda_{1,l_{1,i}+1}^{{(i)}})^{2}}\sum_{l={l_{1,i}}+1}^{M_{i}}(\frac{\lambda_{1,l}^{{(i)}}}{H})^{2}c_{l}^{2}.
\]
Furthermore,
\[
\sum_{l={l_{1,i}}+1}^{M_{i}}(\frac{\lambda_{1,l}^{{(i)}}}{H})^{2}c_{l}^{2}\leq\sum_{l=1}^{M_{i}}(\frac{\lambda_{1,l}^{{(i)}}}{H})^{2}c_{l}^{2} =  (\widetilde{\kappa})^{-1} \int_{\partial K_{i}}
(\kappa\nabla u_{1}\cdot n_{\partial K_i})^{2}.
\]
Consequently, we obtain the following bound
\[
\|u_{1}-\widehat{{u}}_{1}\|_{\text{DG}}^{2}\leq \sum_{i=1}^{N}\frac{H}{ \widetilde{\kappa} \lambda_{1,l_{1,i}+1}^{{(i)}}}(1+\frac{\gamma H}{h\lambda_{1,l_{1,i}+1}^{{(i)}}})
\int_{\partial K_{i}} (\kappa\nabla u_{1}\cdot n_{\partial K})^{2}.
\]

Next, we will find an estimate of $\|u_{2}-\widehat{u}_{2}\|_{\text{DG}}$.
By definition of the bilinear form $a_{DG}$,
\begin{equation*}
a_{DG}(u_{2},v)  =-a_{DG}(u_{1},v)+(f,v)
  =\sum_{E\in\mathcal{E}^{H}}\int_{E}\Big(\average{{\kappa}\nabla{v}\cdot{n}_{E}}\jump{u_{1}}\Big) + (f,v)
\end{equation*}
which holds for any $v \in V^{\text{snap}}_2$. In addition, by the fact that any function in $V^{\text{snap}}_2$ is zero on boundaries of coarse grid blocks,
we have
\begin{equation*}
\|u_{2}-\widehat{{u}}_{2}\|_{\text{DG}}^{2}  =a_{DG}(u_{2}-\widehat{{u}}_{2},u_{2}-\widehat{{u}}_{2})
  =a_{DG}(u_{2},u_{2}-\widehat{{u}}_{2}),
\end{equation*}
where the last equality follows from the fact that the eigenfunctions of (\ref{eq:spec-dg-1}) are $\kappa$-orthogonal on every coarse grid block.
Therefore we have
\begin{equation}
\|u_{2}-\widehat{{u}}_{2}\|_{\text{DG}}^{2}  =\sum_{E\in\mathcal{E}^{H}}\int_{E}\Big(\average{{\kappa}\nabla({u_{2}-\widehat{{u}}_{2}})\cdot{n}_{E}}\jump{u_{1}}\Big)+(f,u_{2}-\widehat{{u}}_{2}).
\label{eq:v2}
\end{equation}
The second term on the right hand side of (\ref{eq:v2}) can be estimated as
\begin{equation*}
(f, u_2 -\widehat{u}_2) \leq \sum_{K\in\mathcal{T}^H} \| f\|_{L^2(K)} \| \kappa^{\frac{1}{2}} (u_2 - \widehat{u}_2) \|_{L^2(K)}.
\end{equation*}
By (\ref{eq:spec-dg-1}), for every $K_i\in\mathcal{T}^H$, we have
\begin{equation*}
\int_{K_i}\kappa|(u_{2}-\widehat{{u}}_{2})|^{2}  =\sum_{l \geq l_{2.i}+1}d_{l}^{2}\leq \Big(\cfrac{H^{2}}{\lambda_{2,l_{2,i}+1}^{{(i)}}}\Big)\sum_{l=l_{2.i}+1}\cfrac{\lambda_{2,l}^{{(i)}}}{H^{2}}d_{l}^{2}
  =\cfrac{H^{2}}{\lambda_{2,l_{2,i}+1}^{{(i)}}}\int_{K_{i}}\kappa|\nabla(u_{2}-\widehat{{u}}_{2})|^{2}.
\end{equation*}
For the first term on the right hand side of (\ref{eq:v2}), we use inequality (\ref{eq:cont5}) to conclude that
\begin{equation*}
\sum_{E\in\mathcal{E}^{H}}\int_{E}\Big(\average{{\kappa}\nabla({u_{2}-\widehat{{u}}_{2}})\cdot{n}_{E}}\jump{u_{1}}\Big)
\leq C_{\text{inv}} \gamma^{-\frac{1}{2}} \| u_2 - \widehat{u}_2 \|_{\text{DG}} \Big( \sum_{E\in\mathcal{E}^H} \frac{\gamma}{h} \int_E \overline{\kappa} \jump{u_1}^2 \Big)^{\frac{1}{2}}.
\end{equation*}
Consequently, from (\ref{eq:v2}) and the fact that $\jump{u_1}=\jump{u_h}$ for all coarse edges, we obtain the following bound
\begin{equation*}
\|u_{2}-\widehat{{u}}_{2}\|_{\text{DG}}^{2}  \leq C \Big( C^2_{\text{inv}} \sum_{E\in\mathcal{E}^H} \frac{1}{h} \int_E \overline{\kappa} \jump{u_h}^2
+\sum_{K\in\mathcal{T}_{H}}\cfrac{H^{2}}{\lambda_{2,l_{2,i}+1}^{{(i)}}} \|f\|_{L^2(K)}^{2} \Big).
\end{equation*}

Finally, we will prove the required error bound. By coercivity,
\[
\begin{split}a_{0}\|\widehat{u}_{1}+\widehat{u}_{2}-u_{H}\|_{\text{DG}}^{2} & \leq a_{\text{DG}}(\widehat{u}_{1}+\widehat{u}_{2}-u_{H},\widehat{u}_{1}+\widehat{u}_{2}-u_{H})\\
 & =a_{\text{DG}}(\widehat{u}_{1}+\widehat{u}_{2}-u_{H},\widehat{u}_{1}+\widehat{u}_{2}-u_{h})+a_{\text{DG}}(\widehat{u}_{1}+\widehat{u}_{2}-u_{H},u_{h}-u_{H}).
\end{split}
\]
Note that $a_{\text{DG}}(\widehat{u}_{1}+\widehat{u}_{2}-u_{H},u_{h}-u_{H})=0$
since $\widehat{u}_{1}+\widehat{u}_{2}-u_{H}\in V^{\text{off}}$.
Using the above results,
\begin{equation*}
\begin{split}
&\: \|\widehat{u}_{1}+\widehat{u}_{2}-u_{H}\|_{\text{DG}}^{2} \\
\leq &\: C\Big( \sum_{i=1}^{N}\frac{H}{ \widetilde{\kappa} \lambda_{1,l_{1,i}+1}^{{(i)}}}(1+\frac{\gamma H}{h\lambda_{1,l_{1,i}+1}^{{(i)}}})
\int_{\partial K_{i}} (\kappa\nabla u_{1}\cdot n_{\partial K})^{2}
+\sum_{K\in\mathcal{T}_{H}}\cfrac{H^{2}}{\lambda_{2,l_{2,i}+1}^{{(i)}}} \|f\|_{L^2(K)}^{2}
+C^2_{\text{inv}} \sum_{E\in\mathcal{E}^H} \frac{1}{h} \int_E \overline{\kappa} \jump{u_h}^2
\Big).
\end{split}
\end{equation*}
The desired bound is then obtained by the triangle inequality
\[
\|u_{h}-u_{H}\|_{\text{DG}}\leq\|u_{h}-\widehat{u}\|_{\text{DG}}+\|\widehat{u}-u_{H}\|_{\text{DG}},
\]
where $\widehat{u} = \widehat{u}_1 + \widehat{u}_2$.
This completes the proof.

\begin{flushright}
$\square$
\par\end{flushright}


We remark that, the term
\begin{equation}
\sum_{i=1}^{N}\frac{H}{ \widetilde{\kappa} \lambda_{1,l_{1,i}+1}^{{(i)}}}(1+\frac{\gamma H}{h\lambda_{1,l_{1,i}+1}^{{(i)}}})
\int_{\partial K_{i}} (\kappa\nabla u_{1}\cdot n_{\partial K})^{2}
\label{eq:error1}
\end{equation}
corresponds to the error for the space $V^{\text{off}}_1$ and the term
\begin{equation*}
\sum_{K\in\mathcal{T}_{H}}\cfrac{H^{2}}{\lambda_{2,l_{2,i}+1}^{{(i)}}} \|f\|_{L^2(K)}^{2}
\end{equation*}
corresponds to the error for the space $V^{\text{off}}_2$.
Moreover, the term
\begin{equation*}
C^2_{\text{inv}} \sum_{E\in\mathcal{E}^H} \frac{1}{h} \int_E \overline{\kappa} \jump{u_h}^2
\end{equation*}
is the error in the fine grid solution $u_h$. This is the irreducible error,
and an estimate of this can be derived following standard DG frameworks.

\begin{table}[ht]
\begin{centering}
\begin{tabular}{|c|c|}
\hline
h & $\Lambda_{K}^{\text{snap}}$\tabularnewline
\hline
1/48 & 1.0021e+03\tabularnewline
\hline
1/96 & 1.0193e+03\tabularnewline
\hline
1/192 & 1.3094e+03\tabularnewline
\hline
\end{tabular}
\begin{tabular}{|c|c|}
\hline
h & $\Lambda_{K}^{\text{snap}}$\tabularnewline
\hline
1/48 & 7.7650e+05\tabularnewline
\hline
1/96 & 1.6569e+06\tabularnewline
\hline
1/192 & 3.3254e+06\tabularnewline
\hline
\end{tabular}
\par\end{centering}

\caption{Left: oversampling basis, Right: no-oversampling basis\label{tab:compare max eig }}

\label{table: eig max}
\end{table}

\begin{remark}
It is important to note that one can also replace (\ref{eq:maxeig}) by
\begin{equation}
\label{eq:largesteig}
a_{H}^{K}(\widehat{v},\widehat{v})\leq  \Lambda_K^{\text{snap}} \widetilde{\kappa} \int_{\partial K}|v|^{2},
\end{equation}
where $\Lambda_K^{\text{snap}}$ is the largest eigenvalue for the spectral problem (\ref{eq:spec-dg}).
Therefore, (\ref{eq:fluxbound}) can be replaced by
\begin{equation*}
\int_{\partial K}|{\kappa}\nabla{u}\cdot{n}_{\partial K}|^{2}\leq  \Lambda_K^{\text{snap}} \widetilde{\kappa}\, a_{H}^{K}(u,u).
\end{equation*}
By following the above steps, we see that one can choose $\gamma$ in (\ref{eq:error1}) so that
\begin{equation*}
\gamma > C_{\kappa} h \max_{K\subset\mathcal{T}^H} \Lambda_K^{\text{snap}}
\end{equation*}
where the constant $C_{\kappa}$ is defined as
\begin{equation*}
C_{\kappa} = \max_{K\subset\mathcal{T}^H} \frac{ \max_{E\subset\partial K} \overline{\kappa} }{ \min_{E\subset\partial K} \overline{\kappa}}.
\end{equation*}
We remark that this constant $C_{\kappa}$ is order one if we assume that every coarse element
has a high contrast region.

One can take smaller values of $\gamma$ if oversampling is used  (oversampling method is discussed in Section \ref{sec:numerresults}).
The main idea of the oversampling is to choose larger regions for
computing snapshot vectors.
For every coarse block $K_i$, we choose an enlarged region $K_i^+$,
and find oversampling snapshot functions $\psi_k^{i,\text{over}}$ by
solving (\ref{eq:overproblem1}). We have performed numerical
experiments and computed $\Lambda_K^{\text{snap}}$
with and without
oversampling. Denote
$\Lambda_{K^+}^{\text{snap}}$ to be the largest eigenvalue
corresponding to the oversampled problem.
In our numerical results (see Table \ref{table: eig max}), we have
removed  linearly dependent snapshot vectors
with respect to the inner product corresponding to
 $\int_{\partial K} |v|^2$ before computing the largest
eigenvalue. Our numerical results show that one
can have about three orders of magnitude smaller value for
$\Lambda_{K^+}^{\text{snap}}$ compared to
$\Lambda_{K}^{\text{snap}}$.
Moreover,
our numerical results show a weak $h$-dependence
for $\Lambda_{K^+}^{\text{snap}}$ as we decrease $h$, while
$\Lambda_{K}^{\text{snap}}$ behaves as $h^{-1}$
(when no-oversampling is used).

Our error analysis holds when oversampling snapshot space
is used.
The term in (\ref{eq:error1}) will become
\begin{equation}
\sum_{i=1}^{N}\frac{H}{ \widetilde{\kappa} \lambda_{1,l_{1,i}+1}^{{(i)}}}(1+\frac{\alpha C_{\kappa} \max_{K\subset\mathcal{T}^H} \Lambda_{K^+}^{\text{snap}}H}{\lambda_{1,l_{1,i}+1}^{{(i)}}})
\int_{\partial K_{i}} (\kappa\nabla u_{1}\cdot n_{\partial K})^{2}
\end{equation}
when $\gamma=\alpha C_{\kappa} h \max_{K\subset\mathcal{T}^H} \Lambda_{K^+}^{\text{snap}}$. If $\Lambda_{K^+}^{\text{snap}}$ is
a weak function of $h$, e.g., if it is bounded with respect to
$h$, then the terms involving
$\Lambda_{K^+}^{\text{snap}}$ doesnt influence the error
and the error is dominated by the first term.
We emphasize that our discussions in this Remark
 are based on our numerical
studies and their analytical studies are difficult because
it requires interior estimates for solutions. We plan to
study them in future.

\end{remark}

\section{A-posteriori error estimate and adaptive enrichment}

\label{sec:errorindicator}

In this section, we will derive an a-posteriori error indicator for
the error $u_h-u_H$ in energy norm. We will then use the
error indicator to develop an adaptive enrichment algorithm. The a-posteriori
error indicator gives an estimate of the local error on the coarse
grid blocks $K_{i}$, and we can then add basis functions to
improve the solution.
Our indicator consists of two components,
which correspond to the errors made in the spaces $V^{\text{snap}}_1$ and $V^{\text{snap}}_2$.
By using the indicator, one can determine adaptively which space has to be enriched.
This section is devoted to the description of the a-posteriori
error indicator and the corresponding adaptive enrichment algorithm.
The convergence analysis of the method will be given in the Section \ref{sec:analysis}.

Recall that $V^h_{\text{DG}}$ is the fine scale DG finite element space,
and the fine scale solution $u_h$ satisfies
\begin{equation}
a_{\text{DG}}(u_h,v)=(f,v)\quad\text{for all}\,\,\, v\in V^h_{\text{DG}}\label{eq:fine}.
\end{equation}
Moreover, the GMsDGM solution $u_H$ satisfies
\begin{equation}
a_{\text{DG}}(u_H,v)=(f,v)\quad\text{for all}\,\,\, v\in V^{\text{off}}.\label{eq:coarse}
\end{equation}
We remark that $V^{\text{off}}\subset V_{\text{DG}}^h$. Next we will give the definitions
of the residuals.

\textbf{Definitions of residuals}:

Let $K_{i}$ be a given coarse grid block.
We will define two residuals corresponding to the two types of snapshot spaces.
First, on the space $V^{i,\text{snap}}_1$, we define the following linear functional
\begin{equation}
R_{1,i}(v)=\int_{K_{i}}fv - a_{\text{DG}}(u_{H},v), \quad v\in V^{i,\text{snap}}_1.
\end{equation}
Similarly,
on the space $V^{i,\text{snap}}_2$, we define the following linear functional
\begin{equation}
R_{2,i}(v)=\int_{K_{i}}fv - a_{\text{DG}}(u_{H},v), \quad v\in V^{i,\text{snap}}_2.
\end{equation}
These residuals measure how well the solution $u_H$ satisfies the fine-scale equation (\ref{eq:fine}).
Furthermore, on the snapshot spaces $V^{i,\text{snap}}_1$ and $V^{i,\text{snap}}_2$, we define
the following norms
\begin{equation}
\|v\|_{V_{1}(K_{i})}^{2}  =H^{-1}\int_{\partial K_{i}}\widetilde{\kappa}v^{2}
\quad\text{and}\quad
\|v\|_{V_{2}(K_{i})}^{2}  =H^{-2}\int_{K_{i}}\kappa v^{2}
\end{equation}
respectively.
The norms of the linear functionals $R_{1,i}$ and $R_{2,i}$ are defined in the standard way, namely
\begin{equation}
\|R_{1,i}\|=\sup_{v\in V^{i,\text{snap}}_1} \frac{|R_{1,i}(v)|}{\|v\|_{V_{1}(K_{i})}}
\quad\text{and}\quad
\|R_{2,i}\|=\sup_{v\in V^{i,\text{snap}}_2} \frac{|R_{2,i}(v)|}{\|v\|_{V_{2}(K_{i})}}.
\end{equation}
The norms $\|R_{1,i}\|$ and $\|R_{2,i}\|$ give estimates on the sizes of fine-scale residual errors
with respect to the spaces $V^{i,\text{snap}}_1$ and $V^{i,\text{snap}}_2$.


We recall that, for each coarse grid block $K_{i}$, the eigenfunctions of the spectral problem (\ref{eq:spec-dg})
corresponding
to the eigenvalues $\lambda_{1,1}^{{(i)}},\cdots,\lambda_{1,l_{1,i}}^{{(i)}}$ and the eigenfunctions of the spectral problem (\ref{eq:spec-dg-1})
corresponding to the eigenvalues
$\lambda_{2,1}^{{(i)}},\cdots,\lambda_{2,l_{2,i}}^{{(i)}}$ are used
in the construction of $V^{\text{off}}$.
In addition, the energy error in this section and Section \ref{sec:analysis}
is measured by $\|u\|_{a}^2 = a_{\text{DG}}(u,u)$,
which is equivalent to the DG norm.

In Section \ref{sec:analysis}, we will prove the following theorem,
and we see that the norms $\|R_{j,i}\|$ give indications
on the size of the energy norm error $\|u_h-u_H\|_{a}$.
\begin{theorem}
Let $u_h$ and $u_{H}$ be the solutions of \eqref{eq:fine}
and \eqref{eq:coarse} respectively. Then
\begin{eqnarray}
\|u_h-u_{H}\|_{a}^{2} & \leq & C_{\text{err}}\sum_{i=1}^{N}\sum_{j=1}^{2}\|R_{j,i}\|^{2}(\lambda_{j,l_{j,i}+1}^{{(i)}})^{-1}.\label{eq:res2}
\end{eqnarray}
\label{thm:post} where $C_{\text{err}}$ is a uniform constant. \end{theorem}


We will now present the adaptive enrichment algorithm. We use $m\geq1$
to represent the enrichment level, $V^{\text{off}}(m)$ to represent the
solution space at level $m$ and $u_H^m$ to represent the GMsDGM solution at the enrichment level $m$.
For each coarse grid block $K_i$, we use $l_{j,i}^{m}$
to represent the number of eigenfunctions in $V^{i,\text{off}}_j$ used at the
enrichment level $m$ for the coarse region $K_{i}$.
Assume that the initial offline space $V^{\text{off}}(0)$
is given.

\textbf{Adaptive enrichment algorithm}: Choose $0<\theta<1$. For
each $m=0,1,\cdots$,
\begin{enumerate}
\item Step 1: Find the solution in the current space. That is, find $u_{H}^{m}\in V^{\text{off}}(m)$
such that
\begin{equation}
a_{\text{DG}}(u_{H}^{m},v)=(f,v)\quad\text{for all}\,\,\, v\in V^{\text{off}}(m).\label{eq:solve}
\end{equation}

\item Step 2: Compute the local residuals. For each coarse grid block $K_{i}$,
we compute
\[
\eta_{j,i}^{2}=\|R_{j,i}\|^{2}(\lambda_{j,l_{j,i}^m+1}^{{(i)}})^{-1}, \quad j=1,2.
\]
Then we
re-enumerate the above $2N$ residuals in the decreasing order, that is, $\eta_{1}^{2}\geq\eta_{2}^{2}\geq\cdots\geq\eta_{2N}^{2}$,
where we adopted single index notations.

\item Step 3: Find the coarse grid blocks and spaces where enrichment is needed. We choose
the smallest integer $k$ such that
\begin{equation}
\theta\sum_{J=1}^{2N}\eta_{J}^{2}\leq\sum_{J=1}^{k}\eta_{J}^{2}.\label{eq:criteria}
\end{equation}

\item Step 4: Enrich the space. For each $J=1,2,\cdots,k$, we add basis
function in $V^{i,\text{off}}_j$ according to the following rule. Let $s$
be the smallest positive integer such that $\lambda_{j,l^m_{j,i}+s+1}^{(i)}$
is large enough (see the proof of Theorem \ref{thm:conv}) compared
with $\lambda_{j,l_{j,i}^m+1}^{(i)}$. Then we include the eigenfunctions
in the construction of the basis functions. The resulting space is
denoted as $V^{\text{off}}(m+1)$.
\end{enumerate}
We remark that the choice of $s$ above will ensure the convergence
of the enrichment algorithm, and in practice, the value of $s$ is
easy to obtain. Moreover, contrary to classical adaptive refinement
methods, the total number of basis functions that we can add is bounded
by the dimension of the snapshot space. Thus, the condition \eqref{eq:criteria}
can be modified as follows. We choose the smallest integer $k$ such
that
\[
\theta\sum_{J=1}^{2N}\eta_{J}^{2}\leq\sum_{J\in I}\eta_{J}^{2},
\]
where the index set $I$ is a subset of $\{1,2,\cdots,k\}$.



Finally, we state the convergence theorem.
\begin{theorem} There
are positive constants $\delta,\rho$ and $L_m$ such
that the following contracting property holds
\[
\|u-u_{H}^{m+1}\|_{a}^{2}+\cfrac{1}{\delta L_{m+1}}\sum_{J=1}^{2N}(S_{J}^{m+1})^{2}\leq\varepsilon\Big(\|u-u_{H}^{m}\|_{a}^{2}+\cfrac{1}{\delta L_{m}}\sum_{J=1}^{N}(S_{J}^{m})^{2}\Big).
\]
\label{thm:conv} Note that $0<\varepsilon<1$ and
\[
\varepsilon=1-\cfrac{\theta^2(1-\rho L_m/L_{m+1})}{\theta^2 + C_{\text{err}}\delta L_{m}}.
\]
\end{theorem}

We remark that the precise definitions of $S_J^m$ as well as
the constants $\delta,\rho$
and $L_m$ are given in Section \ref{sec:analysis}.

\section{Numerical Results}
\label{sec:numerresults}

In this section, we will present some numerical examples to demonstrate
the performance of the adaptive
enrichment algorithm. The domain $\Omega$ is taken as the unit square $[0,1]^2$ and is divided
into $16\times 16$ coarse blocks consisting of uniform squares. Each coarse block is then divided into
$32 \times 32$ fine blocks consisting of uniform squares. Consequently, the whole
domain is partitioned by a $512 \times 512$ fine grid blocks. We will use
the following error quantities to compare the accuracy of our algorithm
\[
e_{2}=\cfrac{\|u_{H}-u_{h}\|_{L^{2}(\Omega)}}{\|u_{h}\|_{L^{2}(\Omega)}},\; e_{a}=\cfrac{\sqrt{a_{\text{DG}}(u_{H}-u_{h},u_{H}-u_{h})}}{\sqrt{a_{\text{DG}}(u_{h},u_{h})}}
\]
\[
e_{2}^{\text{snap}}=\cfrac{\|u_{H}-u_{\text{snap}}\|_{L^{2}(\Omega)}}{\|u_{\text{snap}}\|_{L^{2}(\Omega)}},\; e_{a}^{\text{snap}}=\cfrac{\sqrt{a_{\text{DG}}(u_{H}-u_{\text{snap}},u_{H}-u_{\text{snap}})}}{\sqrt{a_{\text{DG}}(u_{\text{snap}},u_{\text{snap}})}}
\]
where $u_{H}$ and $u_{h}$  are the GMsDGM and the
fine grid solutions respectively.
Moreover, $u_{\text{snap}}$ is the snapshot solution obtained by using all snapshot functions
generated by an oversampling strategy, see below.

We consider the permeability field $\kappa$ which is shown in Figure~\ref{fig:kappa}.
The boundary condition is set to be bi-linear, $g=x_1x_2$.
We will consider two examples with
two different source functions $f$. We will compare
the result of $V_{1}$ enrichment, $V_{1}-V_{2}$ enrichment, oversampling
basis enrichment, uniform enrichment and the exact indicator enrichment.
The following gives the details of these enrichments.

\begin{itemize}
\item \textbf{$V_{1}$ enrichment}: We use the error indicator, $\eta_{1,i}^{2}$
to perform the adaptive algorithm by enriching the basis functions
in $V_{1}$ space only, that is, basis functions obtained by the first spectral problem (\ref{eq:spec-dg}).
We use $4$ basis functions from (\ref{eq:spec-dg})
and zero basis function from (\ref{eq:spec-dg-1}) in the initial
step.
\item \textbf{$V_{1}-V_{2}$ enrichment}: We use both the error indicators, $\eta_{1,i}^{2},\eta_{2,i}^{2}$
to perform the adaptive algorithm by enriching the basis functions
in both $V_{1}$ and $V_{2}$ spaces, that is, basis functions from both spectral problems (\ref{eq:spec-dg}) and (\ref{eq:spec-dg-1}).
We use $4$ basis functions from (\ref{eq:spec-dg})
and zero basis function from (\ref{eq:spec-dg-1}) in the initial step.
\item \textbf{Oversampling enrichment:}
For every coarse block $K_i$, we choose an enlarged region $K_i^+$
(in the examples presented below, we enlarge the coarse block in each direction by a length $H$, that is $K_i^+$ is a 3$\times$3 coarse blocks
with $K_i$ at the center).
Then we find oversampling snapshot functions
${\psi}_{k}^{i,\text{over}} \in V_h(K_i^+)$ by solving
\begin{equation}
\begin{split}\int_{K_{i}}\kappa\nabla\psi_{k}^{i,\text{over}} \cdot\nabla v & ={0},\quad\text{\ensuremath{\forall}}\; v\in V^0_{h}(K_{i}^+), \\
\psi_{k}^{i,\text{over}} & =\delta_{k},\quad\text{ on }\;\partial K_{i}^+.
\end{split}
\label{eq:overproblem1}
\end{equation}
The linear span of these snapshot functions is called $V^{i,\text{over}}$.
Then we choose $40$ dominant oversampling basis functions by POD method. Specifically, we solve the following eigenvalue problem
    \[
\int_{K_i}\psi^{i,\text{over}}_{k}v=\lambda^{i}_{k}\int_{\partial K^{+}_i}\psi^{i,\text{over}}_{k}v,\quad\forall v\in V^{i,\text{over}}.
\] and choose the first $40$ eigenfunctions with largest eigenvalues.
Then we use these $40$ functions as boundary conditions in (\ref{eq:snapproblem1})
and repeat the remaining construction of the offline space.

\item \textbf{Uniform} \textbf{enrichment}: We enrich the basis functions
in $V_{1}$ space uniformly with $4$ basis functions from the $V_1$ space in the
initial step.
\item \textbf{The exact indicator} \textbf{enrichment}: We use the exact
error as the error indicator to perform the adaptive algorithm by
enriching the basis functions in $V_{1}$ space only with $4$
basis functions in the space $V_1$ in the initial step.
Here, the exact error is defined as $\|u - u_H\|_a$.
\end{itemize}

\begin{figure}[ht]
\begin{centering}
\includegraphics[scale=0.4]{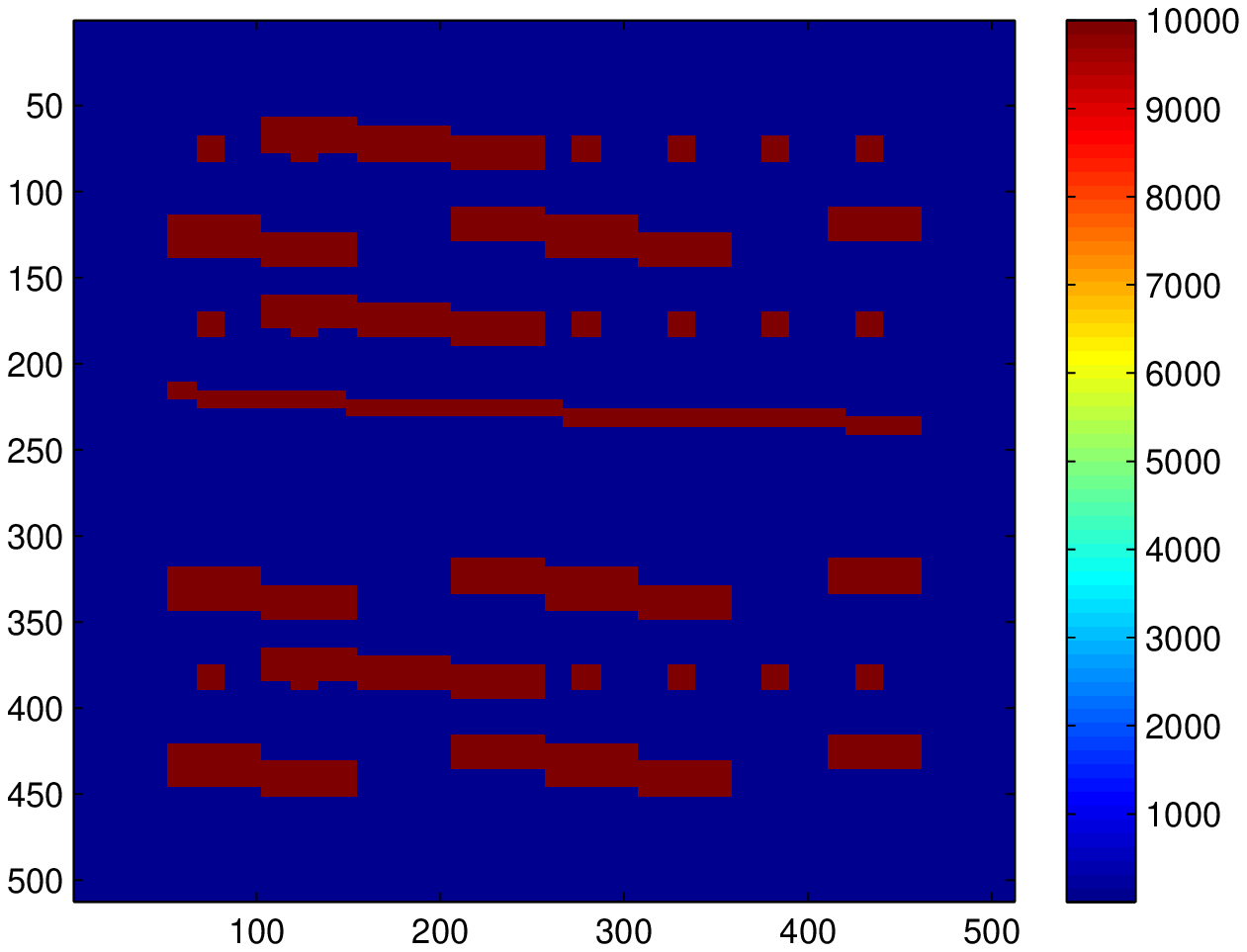}
\includegraphics[scale=0.4]{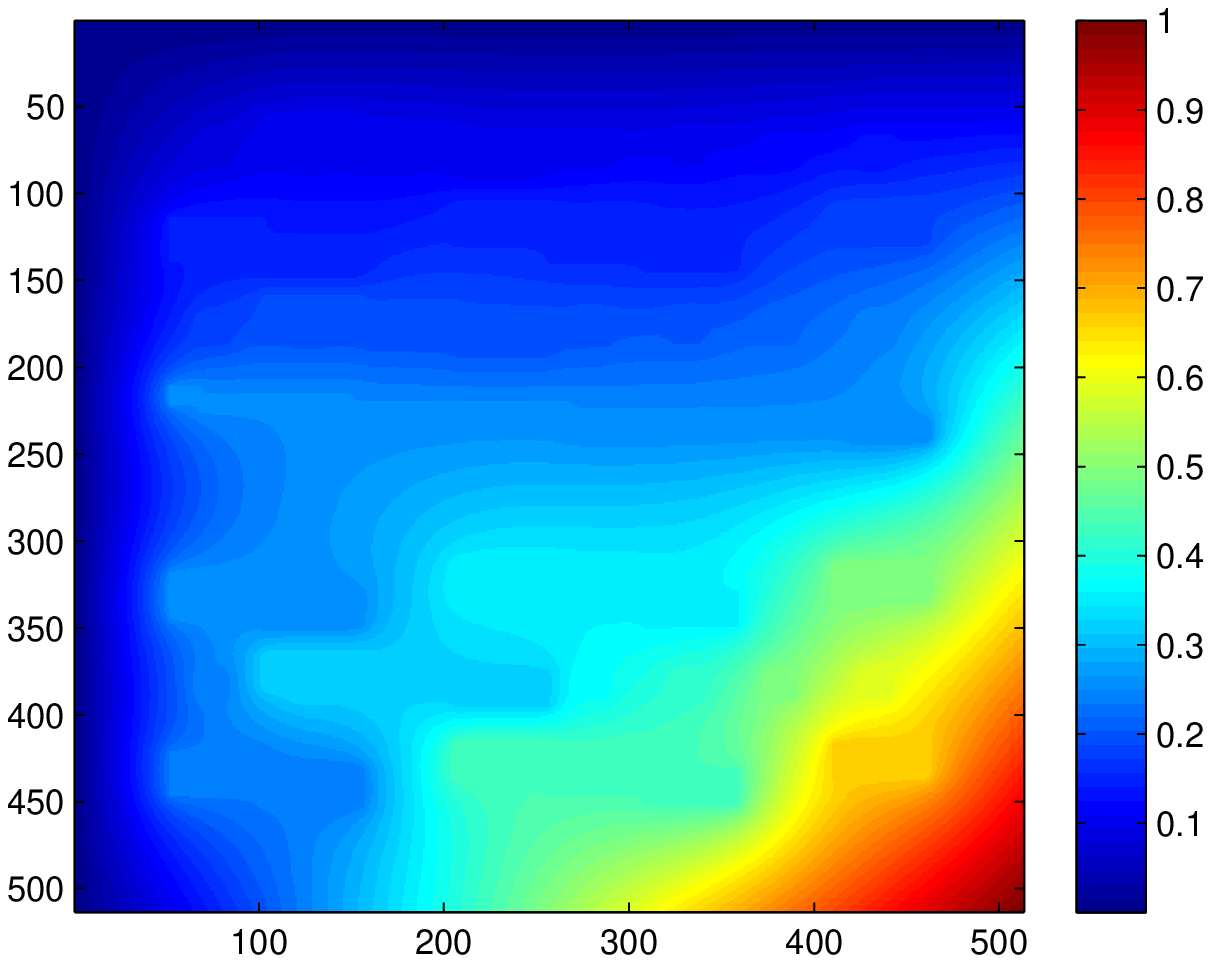}
\par\end{centering}

\caption{Left: Permeability field $\kappa$. Right: Fine grid solution with $f=1$.}
\label{fig:kappa}
\end{figure}

\subsubsection*{Example 1}
In our first example, we take the source function $f=1$.
The fine grid solution is shown in Figure~\ref{fig:kappa}. In Table
\ref{table: case1 V1} and Table \ref{table: case1 oversampling},
we present the convergence history of our algorithm for enriching in $V_{1}$
space only, enriching in both $V_{1}$ and $V_{2}$ spaces and enriching
by the oversampling basis functions.
We remark that, in the presentation of our results, DOF means
the total number of basis functions used in the whole domain.
We see from Table \ref{table: case1 V1} that the behaviour of enriching in $V_1$ space only
and enriching in both $V_1$ and $V_2$ spaces are similar.
The is due to the fact that the source function $f$ is a constant function,
and the space $V_2$ will not help to improve the solution.
This is in consistent with classical theory that basis functions obtained by harmonic extensions are
good enough to approximate the solution.
In Table \ref{table: case1 oversampling}, the convergence behaviour is shown for the oversampling case,
and we see again that a clear convergence is obtained.
For this case, we use $40$ snapshot basis functions per coarse grid block
giving a total DOF of $10240$,
and the corresponding snapshot errors (that is, the difference between the solution obtained by these $10240$ basis functions
and the solution $u_h$)
of $4.5195 \times 10^{-4}$ and $9.8935 \times 10^{-4}$
in relative $L^2$ norm and relative $a$-norm respectively.
In addition, we observe that the oversampling basis provides more efficient representation of the solution
than the non-oversampling basis.
To further demonstrate the
efficiency of our algorithm, we compare our result with the uniform
enrichment scheme. The convergence history for using uniform enrichment is shown
in Table \ref{table: case1 oversampling},
and we see that our adaptive enrichment algorithm performs much better than uniform enrichment.
Finally, a comparison among all the above cases and the enrichment by exact error
is shown in Figure \ref{fig: case1 error plot},
in which the energy error is plotted against DOF.
From the figure, we clearly see that our enrichment algorithm
performs much better than uniform enrichment.
Moreover, our enrichment algorithm performs equally well
compared with enrichment by the exact error. This shows that
our indicator is both reliable and efficient.





\begin{table}[ht]
\begin{centering}
\begin{tabular}{|c|c|c|}
\hline
DOF  & $e_{2}$ & $e_{a}$\tabularnewline
\hline
1024  & 0.1082  & 0.0479\tabularnewline
\hline
1769  & 0.0456  & 0.0178\tabularnewline
\hline
2403  & 0.0156  & 0.0105\tabularnewline
\hline
3135  & 0.0070  & 0.0067\tabularnewline
\hline
5607  & 0.0016  & 0.0031\tabularnewline
\hline
\end{tabular}
\hspace{1cm}
\begin{tabular}{|c|c|c|}
\hline
DOF  & $e_{2}$ & $e_{a}$\tabularnewline
\hline
1024  & 0.1082  & 0.0479\tabularnewline
\hline
1639  & 0.0802  & 0.0239\tabularnewline
\hline
2584  & 0.0194  & 0.0114\tabularnewline
\hline
3822  & 0.0061  & 0.0063\tabularnewline
\hline
5660  & 0.0021  & 0.0037\tabularnewline
\hline
\end{tabular}
\par\end{centering}

\caption{Convergence history with $\theta$=0.4. Left: Enrich in $V_{1}$ space only. Right: Enrich in both $V_1$ and $V_2$ spaces.}

\label{table: case1 V1}
\end{table}




\begin{table}[ht]
\begin{centering}
\begin{tabular}{|c|c|c|c|c|}
\hline
DOF  & $e_{2}$ & $e_{a}$ & $e_{2}^{\text{snap}}$ & $e_{a}^{\text{snap}}$\tabularnewline
\hline
1024  & 0.0940  & 0.0469  & 0.0939  & 0.0469\tabularnewline
\hline
1975  & 0.0204  & 0.0121  & 0.0202  & 0.0121\tabularnewline
\hline
2648  & 0.0087  & 0.0077  & 0.0084  & 0.0077\tabularnewline
\hline
3422  & 0.0046  & 0.0056  & 0.0043  & 0.0056\tabularnewline
\hline
6748  & 0.0009  & 0.0022  & 0.0006  & 0.0020\tabularnewline
\hline
\end{tabular}
\hspace{1cm}
\begin{tabular}{|c|c|c|}
\hline
DOF  & $e_{2}$ & $e_{a}$\tabularnewline
\hline
1024  & 0.1082  & 0.0479\tabularnewline
\hline
2048  & 0.0671  & 0.0199\tabularnewline
\hline
3328  & 0.0423  & 0.0150\tabularnewline
\hline
5888  & 0.0161  & 0.0059\tabularnewline
\hline
8448  & 0.0128  & 0.0044\tabularnewline
\hline
\end{tabular}
\par\end{centering}

\caption{Left: Convergence history for oversampling basis with $\theta=0.4$ and enrichment in $V_{1}$ space only.
Right: Convergence history for uniform enrichment in $V_1$ space only.}

\label{table: case1 oversampling}
\end{table}




\begin{figure}[ht]
\begin{centering}
\includegraphics[scale=0.5]{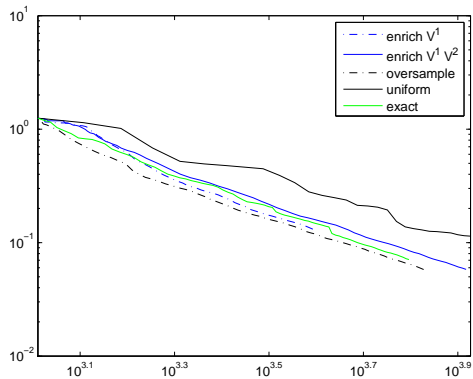}
\par\end{centering}

\caption{A comparison of different ways of enrichment.}

\label{fig: case1 error plot}
\end{figure}




\subsubsection*{Example 2}
In our second example, we will take the source $f$ to be the function shown
in the left plot of Figure \ref{fig:case2 f} and the corresponding fine grid solution shown
in the right plot of Figure \ref{fig:case2 f}. In Table \ref{table: case2 V1}
and Table \ref{table: case2 oversampling}, we present
the convergence history of our algorithm for enriching in $V_{1}$
space only, enriching in both $V_{1}$ and $V_{2}$ spaces and enriching
by oversampling basis.
We see from Table \ref{table: case2 V1} that enrichment in both $V_1$ and $V_2$ spaces
provides much more efficient methods than enrichment in $V_1$ space only.
In particular, for an error level of approximately $1\%$,
we see that enrichment in both $V_1$ and $V_2$ spaces requires $3144$ DOF
while enrichment in $V_1$ space only requires $3483$ DOF.
In Table \ref{table: case2 oversampling}, the convergence behaviour is shown for the oversampling case,
and we see again that a clear convergence is obtained.
For this case, we use $40$ snapshot basis functions per coarse grid block
giving a total DOF of $10240$,
and the corresponding snapshot errors (that is, the difference between the solution obtained by these $10240$ basis functions
and the solution $u_h$)
of $0.0078$ and $0.0093$
in relative $L^2$ norm and relative $a$-norm respectively.
In addition, we observe again that the oversampling basis provides more efficient representation of the solution
than the non-oversampling basis.
To further demonstrate the
efficiency of our algorithm, we compare our results with the uniform
enrichment scheme. The result for using uniform enrichment is shown
in Table \ref{table: case2 oversampling} and we clearly observe that our adaptive method
is more efficient.
Moreover, a comparison of the performance of various strategies is shown in
Figure \ref{fig:case2 error plot},
where the errors against DOF are plotted.
From the figure, we see that our method is much better than uniform enrichment.
Furthermore, enrichment in both $V_1$ and $V_2$ spaces has the best performance,
which suggests that both $V_1$ and $V_2$ spaces are important for more complicated source functions.


\begin{figure}[ht]
\begin{centering}
\includegraphics[scale=0.4]{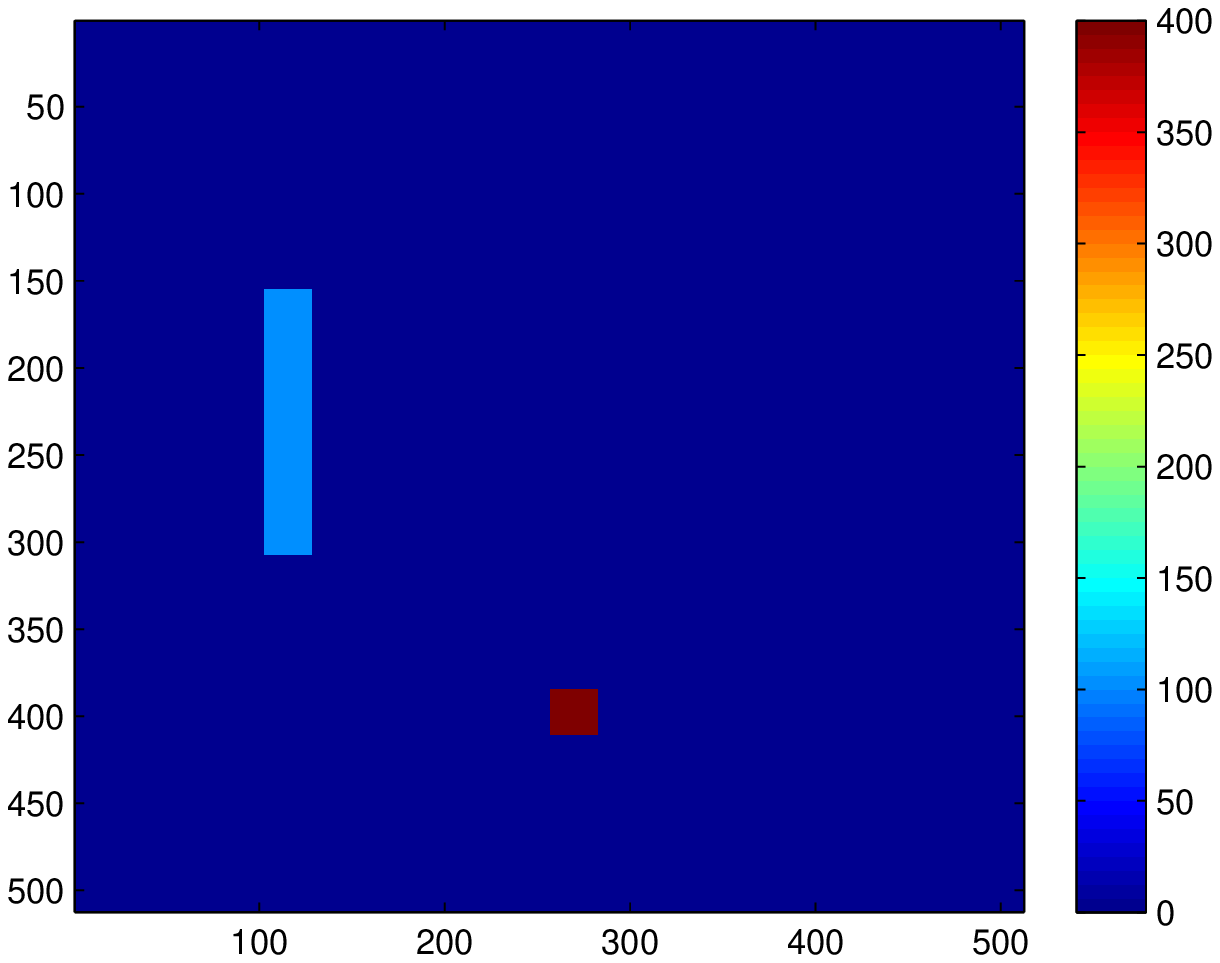}
\includegraphics[scale=0.4]{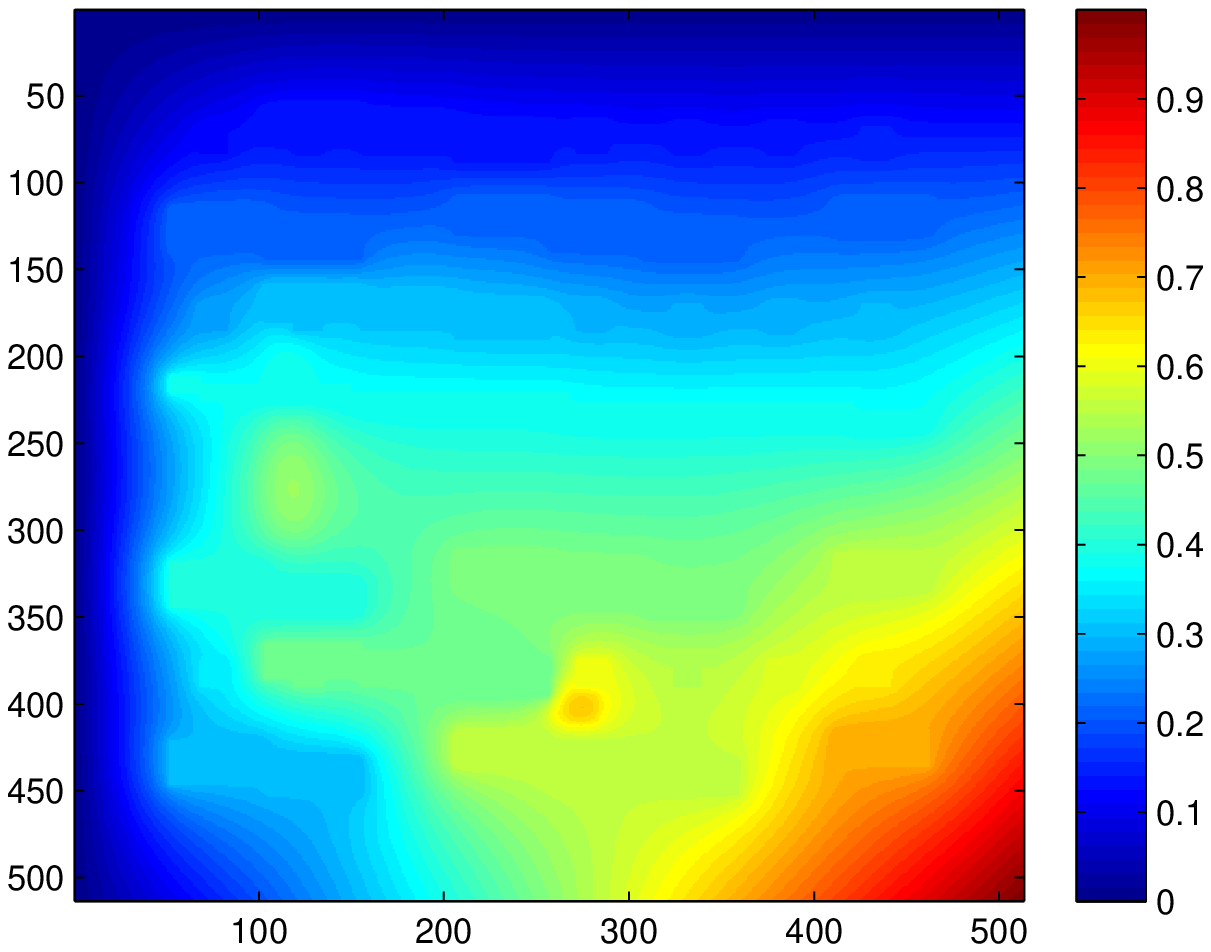}
\par\end{centering}

\caption{Left: The source function $f$ for the second example. Right: The fine grid solution.}

\label{fig:case2 f}
\end{figure}




\begin{table}[ht]
\begin{centering}
\begin{tabular}{|c|c|c|}
\hline
DOF  & $e_{2}$ & $e_{a}$\tabularnewline
\hline
1024  & 0.2052  & 0.0554\tabularnewline
\hline
2028  & 0.0362  & 0.0191\tabularnewline
\hline
2717  & 0.0152  & 0.0140\tabularnewline
\hline
3483  & 0.0111  & 0.0118\tabularnewline
\hline
5116  & 0.0084  & 0.0102\tabularnewline
\hline
\end{tabular}
\hspace{1cm}
\begin{tabular}{|c|c|c|}
\hline
DOF  & $e_{2}$ & $e_{a}$\tabularnewline
\hline
1024  & 0.2052  & 0.0554\tabularnewline
\hline
2023  & 0.0486  & 0.0206\tabularnewline
\hline
3144  & 0.0113  & 0.0105 \tabularnewline
\hline
4456  & 0.0050  & 0.0066\tabularnewline
\hline
7407  & 0.0013  & 0.0034\tabularnewline
\hline
\end{tabular}
\par\end{centering}

\caption{Convergence history with $\theta$=0.4. Left: Enrich in $V_{1}$ space only. Right: Enrich in both $V_1$ and $V_2$ spaces.}

\label{table: case2 V1}
\end{table}




\begin{table}[ht]
\begin{centering}
\begin{tabular}{|c|c|c|c|c|}
\hline
DOF  & $e_{2}$ & $e_{a}$ & $e_{2}^{\text{snap}}$ & $e_{a}^{\text{snap}}$\tabularnewline
\hline
1024  & 0.1882  & 0.0540  & 0.1865  & 0.0532\tabularnewline
\hline
1926  & 0.0296  & 0.0182  & 0.0269  & 0.0156\tabularnewline
\hline
2626  & 0.0137  & 0.0135  & 0.0098  & 0.0098\tabularnewline
\hline
3368  & 0.0105  & 0.0116  & 0.0057  & 0.0070\tabularnewline
\hline
6677  & 0.0080  & 0.0097  & 0.0007  & 0.0025\tabularnewline
\hline
\end{tabular}
\hspace{1cm}
\begin{tabular}{|c|c|c|}
\hline
DOF  & $e_{2}$ & $e_{a}$\tabularnewline
\hline
1024  & 0.2052  & 0.0554\tabularnewline
\hline
2048  & 0.0923  & 0.0282\tabularnewline
\hline
3328  & 0.0659  & 0.0215\tabularnewline
\hline
5888  & 0.0278  & 0.0135\tabularnewline
\hline
8448  & 0.0226  & 0.0121\tabularnewline
\hline
\end{tabular}
\par\end{centering}

\caption{Left: Convergence history for oversampling basis with $\theta=0.4$ and enrichment in $V_{1}$ space only.
Right: Convergence history for uniform enrichment in $V_1$ space only.}

\label{table: case2 oversampling}
\end{table}




\begin{figure}[ht]
\begin{centering}
\includegraphics[scale=0.5]{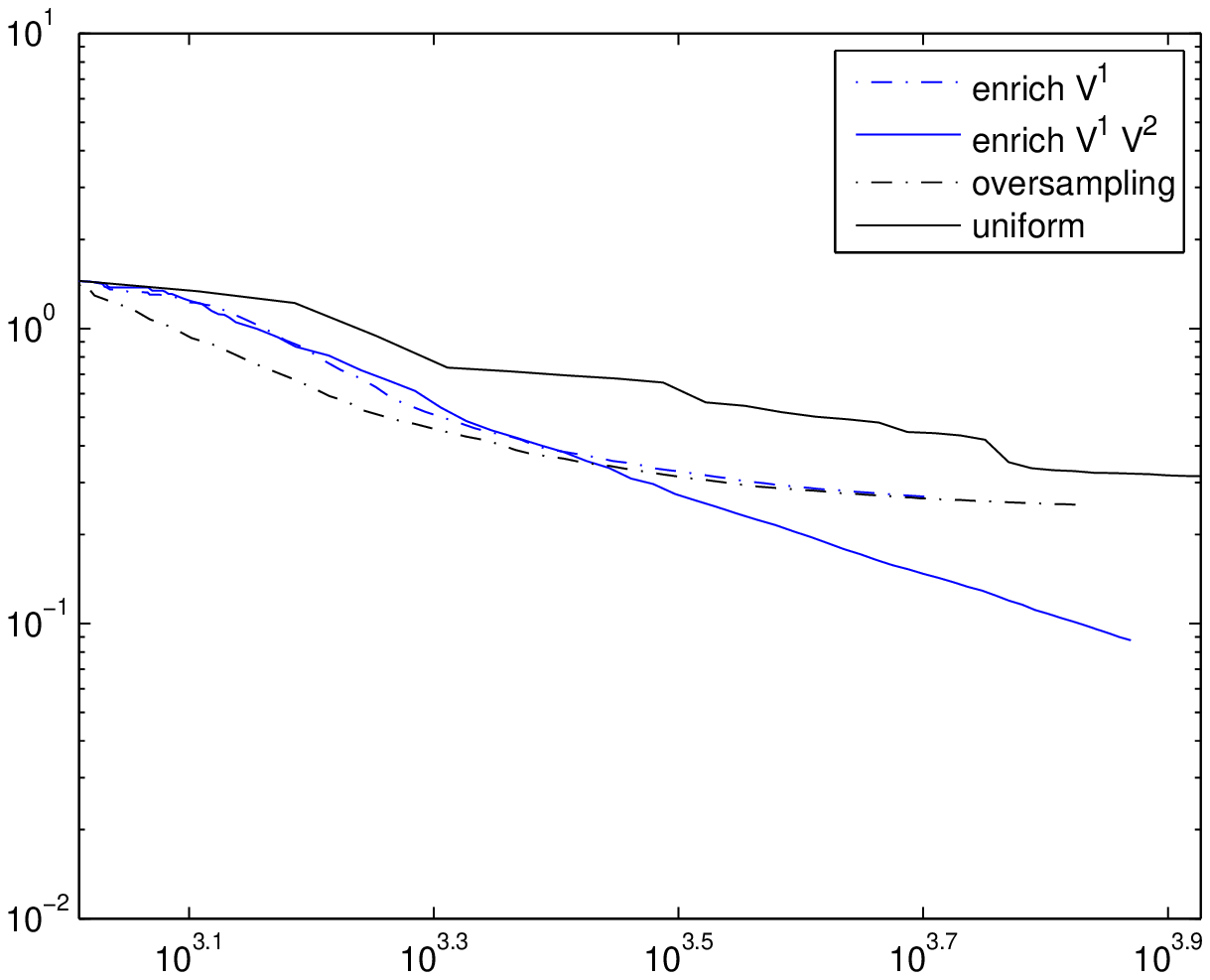}
\par\end{centering}

\caption{A comparison of different ways of enrichment.}

\label{fig:case2 error plot}
\end{figure}






\subsection{Adaptive enrichment algorithm}

\subsubsection{Adaptive enrichment algorithm with basis removal}

In our adaptive enrichment algorithm, we can add basis functions to the offline space
by using the error indicators.
However, the addition of the basis functions must follow the ordering of the eigenfunctions.
There may be cases that some of the intermediate eigenfunctions are not required in the representation of the solution.
Therefore, we propose a numerical strategy to remove basis functions that do not contribute
or contribute less to the representation of the solution.
In the following, we will present this numerical strategy.

\textbf{Adaptive enrichment algorithm with basis removal:} Choose $0<\theta<1$. For
each $m=0,1,\cdots$,
\begin{enumerate}
\item Step 1: Find the solution in the current space. That is, find $u_{H}^{m}\in V^{\text{off}}(m)$
such that
\begin{equation}
a_{\text{DG}}(u_{H}^{m},v)=(f,v)\quad\text{for all}\,\,\, v\in V^{\text{off}}(m).\label{eq:solve1}
\end{equation}

\item Step 2: Compute the local residuals. For each coarse grid block $K_{i}$,
we compute
\[
\eta_{j,i}^{2}=\|R_{j,i}\|^{2}(\lambda_{j,l_{j,i}^m+1}^{{(i)}})^{-1}, \quad j=1,2.
\]
Then we
re-enumerate the $2N$ residuals in the decreasing order, that is, $\eta_{1}^{2}\geq\eta_{2}^{2}\geq\cdots\geq\eta_{2N}^{2}$,
where we adopted single index notations.

\item Step 3: Find the coarse grid blocks where enrichment is needed. We choose
the smallest integer $k$ such that
\begin{equation}
\theta\sum_{J=1}^{2N}\eta_{J}^{2}\leq\sum_{J=1}^{k}\eta_{J}^{2}.\label{eq:criteria1}
\end{equation}

\item Step 4: Enrich the space. For each $J=1,2,\cdots,k$, we add basis
function in $V^{i,\text{off}}_j$ according to the following rule. Let $s$
be the smallest positive integer such that $\lambda_{j,l^m_{j,i}+s+1}^{(i)}$
is large enough (see the proof of Theorem \ref{thm:conv}) compared
with $\lambda_{j,l_{j,i}^m+1}^{(i)}$. Then we include the eigenfunctions
in the construction of the basis functions. The resulting space is
denoted as $\widehat{V}^{\text{off}}(m+1)$.
Note that this is the offline space without basis removal.

\item Step 5: Remove basis. For each coarse grid block $K_i$, we can write the restriction
of the current solution $u_H^m$ on $K_i$ as
\begin{equation*}
 \sum_{l=1}^{l_{1,i}^m} \alpha_{1,l} \phi_l^{(i)} + \sum_{l=1}^{l_{2,i}^m} \alpha_{2,l} \xi_l^{(i)}.
\end{equation*}
Fixed a tolerance $\varepsilon>0$. Then the basis function $\phi_l^{(i)}$ or $\xi_l^{(i)}$ is removed if
\begin{equation*}
\alpha_{1,l}^2 < \varepsilon \Big( \sum_{l=1}^{l_{1,i}^m} \alpha_{1,l}^2 + \sum_{l=1}^{l_{2,i}^m} \alpha_{2,l}^2 \Big)
\quad \text{ or } \quad
\alpha_{2,l}^2 < \varepsilon \Big( \sum_{l=1}^{l_{1,i}^m} \alpha_{1,l}^2 + \sum_{l=1}^{l_{2,i}^m} \alpha_{2,l}^2 \Big)
\end{equation*}
is satisfied. The resulting space is called $V^{\text{off}}(m+1)$.
\end{enumerate}

To test this strategy, we consider our second example with the source function $f$
defined in Figure \ref{fig:case2 f}.
We will consider three choices of $\varepsilon$, with values $10^{-12}$, $10^{-13}$ and $10^{-14}$.
The convergence history of these cases are shown in Table \ref{table: rm case1 tol1e-12}.
We can see that our basis removal strategy gives more efficient representation of the solution.
For example, comparing the errors with DOF of around $2000$
with basis removal (Table \ref{table: rm case1 tol1e-12})
and without basis removal (Table \ref{table: case2 V1}),
we see that the method with basis removal gives a solution with smaller errors
in both $L^2$ norm and $a$-norm.
On the other hand, we see that the choice of $\varepsilon = 10^{-14}$
performs better than $\varepsilon = 10^{-12}$.
In particular, for DOF of around $2200$, the error with $\varepsilon = 10^{-14}$
is around $2\%$ while the error with $\varepsilon = 10^{-12}$
is around $4\%$.
However, one expects that smaller choices of $\varepsilon$
are not as economical as larger choices of $\varepsilon$.




\begin{table}[ht]
\begin{centering}
\begin{tabular}{|c|c|c|}
\hline
DOF & $L_{2}$-error & $a$-error\tabularnewline
\hline
1024 & 0.2052 & 0.0554\tabularnewline
\hline
951 & 0.1824 & 0.0502\tabularnewline
\hline
1074 & 0.1158 & 0.0415\tabularnewline
\hline
1742 & 0.0461 & 0.0174\tabularnewline
\hline
2218 & 0.0404 & 0.0153\tabularnewline
\hline
\end{tabular}
\hspace{1cm}
\begin{tabular}{|c|c|c|}
\hline
DOF & $L_{2}$-error & $a$-error\tabularnewline
\hline
1024 & 0.2052 & 0.0554\tabularnewline
\hline
996 & 0.1767 & 0.0501\tabularnewline
\hline
1107 & 0.1236 & 0.0431\tabularnewline
\hline
2006 & 0.0266 & 0.0154\tabularnewline
\hline
2824 & 0.0192 & 0.0123\tabularnewline
\hline
\end{tabular}
\hspace{1cm}
\begin{tabular}{|c|c|c|}
\hline
DOF & $L_{2}$-error & $a$-error\tabularnewline
\hline
1024 & 0.2052 & 0.0554\tabularnewline
\hline
1048 & 0.1774 & 0.0500\tabularnewline
\hline
1185 & 0.1280 & 0.0434\tabularnewline
\hline
2235 & 0.0223 & 0.0150\tabularnewline
\hline
3275 & 0.0147 & 0.0117\tabularnewline
\hline
\end{tabular}
\par\end{centering}

\caption{Enrichment with $\theta=0.4$ and basis removal as well as  enrichment in $V_{1}$ space only.
Left: $\varepsilon = 10^{-12}$. Middle: $\varepsilon = 10^{-13}$. Right: $\varepsilon = 10^{-14}$}

\label{table: rm case1 tol1e-12}
\end{table}

\subsubsection{Adaptive enrichment using local basis pursuit}

In this section, we discuss an algorithm that follows basis pursuit ideas
\cite{donoho_sparsity_review}
and identify the basis functions which need to be added based on the residual.
The main idea
is to find multiscale basis functions that correlate to the residual the most
and add those basis functions.
More precisely, we identify basis functions that has
the largest correlation
coefficient with respect to the residual and add those
basis functions.
In the following, we will present the details of the numerical algorithm.

\textbf{Adaptive enrichment algorithm using local basis pursuit:} Choose $0<\theta<1$. For
each $m=0,1,\cdots$,
\begin{enumerate}
\item Step 1: Find the solution in the current space. That is, find $u_{H}^{m}\in V^{\text{off}}(m)$
such that
\begin{equation}
a_{\text{DG}}(u_{H}^{m},v)=(f,v)\quad\text{for all}\,\,\, v\in V^{\text{off}}(m).\label{eq:solve2}
\end{equation}

\item Step 2: Compute the local residuals. For each coarse grid block $K_{i}$,
we compute
\[
\zeta_{j,i,l}^{2}=\cfrac{|R_{j,i}(v_l)|^{2}}{\|v_l\|^2_{V_j (K_i)}}, \quad j=1,2, \quad \forall v_l \in V_{j}(K_i).
\]
Then we
re-enumerate these residuals in the decreasing order, that is, $\zeta_{1}^{2}\geq\zeta_{2}^{2}\geq\cdots$,
where we adopted single index notations. Note that
$|R_{j,i}(v_l)|$ is the inner-product that identifies the basis functions
that have the largest correlation to the residual. More precisely,
\[
|R_{j,i}(v_l)|=|\int_{K_i} fv_{l} - a_{DG}(u^m_H,v_{l})|
\]
which is just the local inner-product of the residual vector and basis function $v_l$.
\item Step 3: Find the coarse grid blocks where enrichment is needed. We choose
the smallest integer $k$ such that
\begin{equation}
\eta_{k}\geq \theta \eta_{1}.\label{eq:criteria12}
\end{equation}

\item Step 4: Enrich the space. For each $J=1,2,\cdots,k$, we add the basis function $v_l \in V_j(K_i)$ corresponding to $\zeta _J$. The resulting space is
denoted as $\widehat{V}^{\text{off}}(m+1)$.
Note that this is the offline space without basis removal.

\item Step 5: Remove basis. For each coarse grid block $K_i$, we can write the restriction
of the current solution $u_H^m$ on $K_i$ as
\begin{equation*}
 \sum_{l=1}^{l_{1,i}^m} \alpha_{1,l} \phi_l^{(i)} + \sum_{l=1}^{l_{2,i}^m} \alpha_{2,l} \xi_l^{(i)}.
\end{equation*}
Fixed a tolerance $\varepsilon>0$. Then the basis function $\phi_l^{(i)}$ or $\xi_l^{(i)}$ is removed if
\begin{equation*}
\alpha_{1,l}^2 < \varepsilon \Big( \sum_{l=1}^{l_{1,i}^m} \alpha_{1,l}^2 + \sum_{l=1}^{l_{2,i}^m} \alpha_{2,l}^2 \Big)
\quad \text{ or } \quad
\alpha_{2,l}^2 < \varepsilon \Big( \sum_{l=1}^{l_{1,i}^m} \alpha_{1,l}^2 + \sum_{l=1}^{l_{2,i}^m} \alpha_{2,l}^2 \Big)
\end{equation*}
is satisfied. The resulting space is called $V^{\text{off}}(m+1)$.
\end{enumerate}

To demonstrate the performance of this strategy, we will consider two examples. In the first example, the source function $f$ is defined as in Figure \ref{fig:case2 f} and the rest of the parameters as in the Example 2. In the second example, we will take the solution  (see Figure \ref{fig:sparse sol}) which only contain the component of the 1st, 17th and 30th eigen-basis. The boundary conditions  are as in Example 2 and the source term is calculated based on this
sparse solution.
The convergence history for the first example
is shown in Table \ref{table: BP case1}. Comparing these results to Table
\ref{table: rm case1 tol1e-12}, we can see that the adaptive enrichment provides a better convergence.
The convergence history is substantially improved if we consider
the sparse solution as in our second example. The numerical
results are shown in  Table \ref{table: BP case2}.

\begin{table}[ht]
\begin{centering}
\begin{tabular}{|c|c|c|}
\hline
DOF  & $e_{2}$ & $e_{a}$\tabularnewline
\hline
1024 & 0.2052 & 0.0554\tabularnewline
\hline
1036 & 0.1474 & 0.0405\tabularnewline
\hline
1259 & 0.0585 & 0.0230\tabularnewline
\hline
2096 & 0.0129 & 0.0125\tabularnewline
\hline
2643 & 0.0099 & 0.0111\tabularnewline
\hline
\end{tabular}
\par\end{centering}

\caption{Enrichment using basis pursuit with $\theta=0.8$ and basis removal as well as  enrichment in $V_{1}$ space only. }

\label{table: BP case1}
\end{table}

\begin{figure}[ht]
\begin{centering}
\includegraphics[scale=0.4]{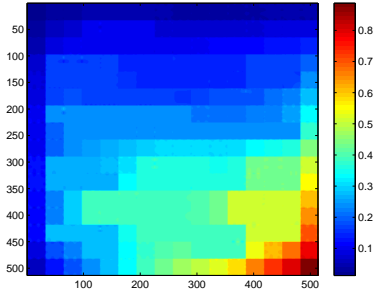}
\par\end{centering}

\caption{Solution with sparse coefficient}

\label{fig:sparse sol}
\end{figure}

\begin{table}[ht]
\begin{centering}
\begin{tabular}{|c|c|c|}
\hline
DOF  & $e_{2}$ & $e_{a}$\tabularnewline
\hline
1024 & 0.0150 & 0.0424\tabularnewline
\hline
941 & 0.0069 & 0.0286\tabularnewline
\hline
934 & 0.0032 & 0.0135\tabularnewline
\hline
688 & 0.0001 & 0.0010\tabularnewline
\hline
744 & 1.20e-06 & 2.13e-05\tabularnewline
\hline
\end{tabular} %
\begin{tabular}{|c|c|c|}
\hline
DOF  & $e_{2}$ & $e_{a}$\tabularnewline
\hline
1024 & 0.0150 & 0.0424\tabularnewline
\hline
997 & 0.0150 & 0.0424\tabularnewline
\hline
1327 & 0.0108 & 0.0412\tabularnewline
\hline
2447 & 0.0023 & 0.0154\tabularnewline
\hline
889 & 0.0003 & 0.0022\tabularnewline
\hline
\end{tabular}
\par\end{centering}

\caption{Enrichment with $\theta$=0.8 . Left: basis pursuit. Right: standard
enrichment}

\label{table: BP case2}
\end{table}

\section{Convergence analysis}

\label{sec:analysis}

In this section, we will provide the proofs for the a-posteriori error
estimates (Theorem \ref{thm:post}) and the convergence of the adaptive enrichment
algorithm (Theorem \ref{thm:conv}). 

For each coarse grid block $K_i$,
$i=1,2,\cdots,N$, we define two projection operators $P_{j,i}:V^{i,\text{snap}}_j \rightarrow V^{i,\text{off}}_j$, $j=1,2$,
from the local snapshot spaces to the corresponding local offline spaces
by
\begin{align*}
\int_{\partial K_{i}}\widetilde{\kappa}P_{1,i}(v)w & =\int_{\partial K_{i}}\widetilde{\kappa}vw\quad\quad \forall w\in V^{i,\text{off}}_1, \\
\int_{K_{i}}\kappa P_{2,i}(v)w & =\int_{K_{i}}\kappa vw\quad\quad\quad \forall w\in V^{i,\text{off}}_2.
\end{align*}
 For any $v\in V^{i,\text{snap}}_1$, we can write
 $v=\sum_{l=1}^{M_{i}}c_{1,l}\phi_{l}^{(i)}$.
 By orthogonality of eigenfunctions, we have $P_{1,i}(v)=\sum_{l=1}^{l_{1,i}}c_{1,l}\phi_{l}^{(i)}$.
Therefore, by the equivalence of $\|\cdot\|_a$ and $\|\cdot \|_{\text{DG}}$, we have
\begin{equation*}
\|P_{1,i}(v)\|_{a}^{2}  \leq a_1\left(\int_{K_{i}}\kappa|\nabla P_{1,i}(v)|^{2}+\cfrac{\gamma}{h}\int_{\partial K_{i}}\widetilde{\kappa}P_{1,i}(v)^{2}\right).
\end{equation*}
By the spectral problem (\ref{eq:spec-dg}) and the fact that the eigenvalues are ordered increasingly, we have
\begin{align*}
 \|P_{1,i}(v)\|_{a}^{2}& \leq a_1\left(\sum_{l=1}^{l_{1,i}}\cfrac{\lambda_{{1,l}}^{{(i)}}}{H}c_{1,l}^{2}+\cfrac{\gamma}{h}\int_{\partial K_{i}}\widetilde{\kappa}P_{1,i}(v)^{2}\right)\\
 & \leq a_1\left(\cfrac{\lambda_{{1,l_{1,i}}}^{{(i)}}}{H}+\cfrac{\gamma}{h}\right)\sum_{l=1}^{l_{1,i}}c_{1,l}^{2}=a_1\left(\lambda_{{1,l_{1,i}}}^{{(i)}}+\cfrac{\gamma H}{h}\right)\|v\|^2_{V_{1}(K_{i})}.
\end{align*}
Similarly, for $v=\sum_{l \geq 1}c_{2,l}\xi_{l}^{(i)}$, we have $P_{2,i}(v)=\sum_{l=1}^{l_{2,i}}c_{2,l}\xi_{l}^{(i)}$.
Therefore, by the equivalence of $\|\cdot\|_a$ and $\|\cdot \|_{\text{DG}}$, we have
\begin{equation*}
\|P_{2,i}(v)\|_{a}^{2}  \leq a_1\left(\int_{K_{i}}\kappa|\nabla P_{2,i}(v)|^{2}\right).
\end{equation*}
By the spectral problem (\ref{eq:spec-dg-1}) and the fact that the eigenvalues are ordered increasingly, we have
\begin{equation*}
\|P_{2,i}(v)\|_{a}^{2}  \leq a_1\left(\sum_{l=1}^{l_{2,i}}\cfrac{\lambda_{{2,l}}^{{(i)}}}{H^{2}}c_{2,l}^{2}\right)
  \leq a_1\left(\cfrac{\lambda_{{2,l_{2,i}}}^{{(i)}}}{H^{2}}\right)\sum_{l=1}^{l_{2,i}}c_{2,l}^{2}= a_1 \lambda_{{2,l_{2,i}}}^{{(i)}} \|v\|^2_{V_{2}(K_{i})}.
\end{equation*}
Thus the projections $P_{j,i}$ satisfy the following stability bound
\begin{equation}
||P_{j,i}(v)||_{a}\leq a_1^{\frac{1}{2}}\left(\lambda_{{j,l_{j,i}}}^{{(i)}}+\cfrac{\gamma H}{h}\right)^{\frac{1}{2}}\|v\|_{V_{j}(K_{i})}, \quad j=1,2, \quad i=1,2,\cdots, N.
\label{eq:eigenstab}
\end{equation}

Next, we will establish some approximation properties for the projection operators $P_{j,i}$.
Indeed, by the definitions of the operators $P_{j,i}$, for any $v\in V^{i,\text{snap}}_j$,
\begin{equation*}
\|v-P_{j,i}(v)\|_{V_{j}(K_{i})}^{2}  =H^{-j}\sum_{l\geq l_{j,i}+1}c_{j,l}^{2}
  \leq(\lambda_{l_{j,i}+1}^{{(i)}})^{-1}\sum_{l \geq l_{j,i}+1}\cfrac{\lambda_{j,l}^{{(i)}}}{H^{j}}c_{j,l}^{2}
  =(\lambda_{{j,l_{j,i}+1}}^{{(i)}})^{-1} \int_{K_{i}} \kappa |\nabla v|^{2},
\end{equation*}
and therefore the following convergence result holds
\begin{align}
\|v-P_{j,i}(v)\|_{V_{j}(K_{i})} & \leq \Big(\lambda_{{j,l_{j,i}+1}}^{{(i)}}\Big)^{-\frac{1}{2}} \Big(\int_{K_{i}} \kappa |\nabla v|^{2}\Big)^{\frac{1}{2}}.
\label{eq:eigenbound}
\end{align}
For the analysis presented below, we
define the projection $\Pi:V^{\text{snap}}\rightarrow V^{\text{off}}$ by
$\Pi v=\sum_{i=1}^{N}\sum_{j=1}^{2}P_{j,i}(v)$.

\subsection{Proof of Theorem \ref{thm:post}}

Let $v\in V_{\text{DG}}^h$ be an arbitrary function in the space $V_{\text{DG}}^h$. Using \eqref{eq:fine},
we have
\[
a_{\text{DG}}(u_h-u_{H},v)=a_{\text{DG}}(u_h,v)-a_{\text{DG}}(u_{H},v)=(f,v)-a_{\text{DG}}(u_{H},v).
\]
Therefore,
\[
a_{\text{DG}}(u_h-u_{H},v)=(f,v)-a_{\text{DG}}(u_{H},v)=(f,v-\Pi v)+(f,\Pi v)-a_{\text{DG}}(u_{H},\Pi v)-a_{\text{DG}}(u_{H},v-\Pi v).
\]
Thus, using \eqref{eq:coarse}, we have
\begin{equation}
a_{\text{DG}}(u_h-u_{H},v)=(f,v-\Pi v)-a_{\text{DG}}(u_{H},v-\Pi v).\label{eq:err1}
\end{equation}
Since the space $V_{\text{DG}}^h$ is the same as $V^{\text{snap}}$, we can write
$v = \sum_{i=1}^N \sum_{j=1}^2 v_j^{(i)}$ with $v_j^{(i)} \in V^{i,\text{snap}}_j$.
Hence, (\ref{eq:err1}) becomes
\begin{equation}
a_{\text{DG}}(u_h-u_{H},v)=\sum_{i=1}^{N}\sum_{j=1}^{2}\Big(\int_{K_{i}}f(v^{(i)}_{j}-P_{j,i}v^{(i)}_{j})-a_{\text{DG}}(u_{H},v^{(i)}_{j}-P_{j,i}v^{(i)}_{j})\Big).\label{eq:err2}
\end{equation}
We remark that, in the computation of the term $a_{\text{DG}}(u_{H},v^{(i)}_{j}-P_{j,i}v^{(i)}_{j})$ in
(\ref{eq:err2}), we assume that the second argument is zero outside the coarse grid block $K_i$.

Using the definition of $R_{j,i}$, we see that \eqref{eq:err2} can
be written as
\[
a_{\text{DG}}(u_h-u_{H},v)=\sum_{i=1}^{N}\sum_{j=1}^{2}R_{j,i}(v^{(i)}_{j}-P_{j,i}v^{(i)}_{j}).
\]
Thus, we have
\[
a_{\text{DG}}(u_h-u_{H},v)\leq\sum_{i=1}^{N}\sum_{j=1}^{2}\|R_{j,i}\| \|(v^{(i)}_{j}-P_{j,i}v^{(i)}_j)\|_{V_{j}(K_{i})}.
\]
Using \eqref{eq:eigenbound},
\[
a_{\text{DG}}(u_h-u_{H},v)\leq\sum_{i=1}^{N}\sum_{j=1}^{2}\|R_{j,i}\| \Big(\lambda_{j,l_{j,i}+1}^{{(i)}}\Big)^{-\frac{1}{2}} \Big(\int_{K_{i}} \kappa |\nabla v^{(i)}_{j}|^2\Big)^{\frac{1}{2} }.
\]
The inequality \eqref{eq:res2} is then followed by taking $v=u_h-u_{H}$
and $\sum_{i=1}^{N}\sum_{j=1}^{2}\int_{K_{i}} \kappa |\nabla v^{(i)}_{j}|^2\leq \| v\|_{\text{DG}}^{2} \leq a_0\|v\|_{a}^{2}$.

\subsection{An auxiliary lemma}

In this section, we will derive an auxiliary lemma which will be used
for the proof of the convergence of the adaptive enrichment algorithm
stated in Theorem \ref{thm:conv}. We use the notation $P_{j,i}^{m}$
to denote the projection operator $P_{j,i}$ at the enrichment level
$m$.


In Theorem \ref{thm:post}, we see that $\|R_{j,i}\|$ gives an upper
bound of the energy error $\|u_h-u_{H}\|_{a}$. We will first
show that, $\|R_{j,i}\|$ is also a lower bound up to a correction
term (see Lemma \ref{lem:recurR}).
To state this precisely, we define
\begin{equation}
S_{j,i}^{m}=(\lambda_{j,l_{j,i}^m+1}^{{(i)}})^{-\frac{1}{2}}  \sup_{v\in V^{i,\text{snap}}_j}\frac{|R_{j,i}(v-P_{j,i}^{m}(v))|}{\|v\|_{V_{j}(K_{i})}}. \label{eq:defSR}
\end{equation}
Notice that
the residual $R_{j,i}$ is computed using the solution $u_{H}^{m}$
obtained at enrichment level $m$. We omit the index $m$ in $R_{j,i}$
to simplify notations. Next, we will obtain
\begin{equation}
(S_{j,i}^{m})^{2} =
\|R_{j,i}\|^{2}(\lambda_{j,l_{j,i}^m+1}^{{(i)}})^{-1}. \label{eq:Rbound}
\end{equation}
Indeed, by the fact that $P_{j,i}^m(v) \in V^{i,\text{off}}_j$,
\begin{equation*}
R_{j,i}(P_{j,i}^{m}(v))  =\int_{K_{i}}f P_{j,i}^{m}(v)-a_{\text{DG}}(u_{H}^{m},P_{j,i}^{m}(v)) = 0.
\end{equation*}
Thus,
\begin{equation*}
S_{j,i}^{m}=(\lambda_{j,l_{j,i}^m+1}^{{(i)}})^{-\frac{1}{2}}  \sup_{v\in V^{i,\text{snap}}_j}\frac{|R_{j,i}(v-P_{j,i}^{m}(v))|}{\|v\|_{V_{j}(K_{i})}}
=(\lambda_{j,l_{j,i}^m+1}^{{(i)}})^{-\frac{1}{2}}  \sup_{v\in V^{i,\text{snap}}_j}\frac{|R_{j,i}(v)|}{\|v\|_{V_{j}(K_{i})}}
=(\lambda_{j,l_{j,i}^m+1}^{{(i)}})^{-\frac{1}{2}}  \| R_{j,i} \|.
\end{equation*}
This implies (\ref{eq:Rbound}).


To prove Theorem \ref{thm:conv}, we will need the following recursive
properties for $S_{j,i}^{m}$ (see Lemma \ref{lem:recurR}).
Notice that, the notation $\|u\|_{a,K_i}$ is defined as
\[
\|u\|_{a,K_i}^2=
a_{\text{DG}}({u},{u})=\int_{K_i} \kappa |\nabla u|^2
-\sum_{E\in\partial K_i}2\int_{E}\average{{\kappa}\nabla{u}\cdot{n}_{E}}\jump{u}+\sum_{E\in\partial K_i}\frac{\gamma}{h}\int_{E}\overline{\kappa}\jump{u}^2.
\]

\begin{lemma} \label{lem:recurR} For any $\alpha>0$, we have
\begin{equation}
(S_{j,i}^{m+1})^{2}\leq(1+\alpha)\frac{\lambda_{j,l_{j,i}^m+1}^{{(i)}}}{\lambda_{j,l_{j,i}^{m+1}+1}^{{(i)}}}(S_{j,i}^{m})^{2}+(1+\alpha^{-1}) a_1D\|u_{H}^{m+1}-u_{H}^{m}\|_{a,K_{i}}^{2}\label{eq:Sbound}
\end{equation}
where the enrichment level dependent constant $D$ is defined by
by
\[
D=\left(\cfrac{\Lambda_{j,i}}{\lambda_{j,l_{j,i}^{m+1}+1}^{{(i)}}}+\cfrac{\gamma H}{h\lambda_{j,l_{j,i}^{m+1}+1}^{{(i)}}}\right)
\]
with $\Lambda_{j,i} = \max_{l} \lambda_{j,l}^{(i)}$.

\end{lemma}

\begin{proof} By direct calculations, we have
\begin{equation}
\begin{split} & \:\int_{K_{i}}f(v-P_{j,i}^{m+1}(v))-a_{\text{DG}}(u_{H}^{m+1},v-P_{j,i}^{m+1}(v))\\
= & \:\int_{K_{i}}fv-a_{\text{DG}}(u_{H}^{m+1},v)\\
= & \:\int_{K_{i}}fv-a_{\text{DG}}(u_{H}^{m},v)+a_{\text{DG}}(u_{H}^{m}-u_{H}^{m+1},v)\\
= & \:\int_{K_{i}}f(v-P_{j,i}^{m}(v))-a_{\text{DG}}(u_{H}^{m},v-P_{j,i}^{m}(v))+a_{\text{DG}}(u_{H}^{m}-u_{H}^{m+1},v).
\end{split}
\label{err}
\end{equation}
By definition of $S_{j,i}^{m}$, we have
\begin{equation}
S_{j,i}^{m}=(\lambda_{j,l_{j,i}^m+1}^{{(i)}})^{-\frac{1}{2}}\sup_{v\in V^{i,\text{snap}}_j}\frac{|\int_{K_{i}}f(v-P_{j,i}^{m}(v))-a_{\text{DG}}(u_{H}^{m},v-P_{j,i}^{m}(v))|}{\|v\|_{V_{j}(K_{i})}}.
\end{equation}
Multiplying \eqref{err} by $(\lambda_{j,l_{j,i}^m+1}^{{(i)}})^{-\frac{1}{2}} \|v\|_{V_{j}(K_{i})}^{-1}$
and taking supremum with respect to $v$, we have
\begin{equation}
S_{j,i}^{m+1}\leq \Big( \frac{\lambda_{j,l_{j,i}^m+1}^{{(i)}}}{\lambda_{j,l_{j,i}^{m+1}+1}^{{(i)}}} \Big)^{\frac{1}{2}}S_{j,i}^{m}+I, \label{err1}
\end{equation}
where
\[
I= (\lambda_{j,l_{j,i}^m+1}^{{(i)}})^{-\frac{1}{2}} \sup_{v\in V^{i,\text{snap}}_j} \frac{|a_{\text{DG}}(u_{H}^{m}-u_{H}^{m+1},v)|}{\|v\|_{V_{j}(K_{i})}}.
\]

To estimate $I$, we note that
\begin{equation*}
a_{\text{DG}}(u_{H}^{m},P_{j,i}^{m}(v))=\int_{K_{i}}f P_{j,i}^{m}(v)=a_{\text{DG}}(u_{H}^{m+1},P_{j,i}^{m}(v)).
\end{equation*}
Therefore, we have
\begin{equation*}
a_{\text{DG}}(u_H^m-u_H^{m+1},v) = a_{\text{DG}}(u_H^m-u_H^{m+1},v - P_{j,i}^m(v))
\leq \| u_H^m - u_H^{m+1} \|_{a,K_i} \| v - P_{j,i}^m(v) \|_{a,K_i},
\end{equation*}
where we remark that $v$ has value zero outside $K_i$.
By the stability bound (\ref{eq:eigenstab}),
\begin{equation*}
\| v-P_{j,i}^m(v)  \|_{a}\leq a_1^{\frac{1}{2}}\left( \Lambda_{j,i}+\cfrac{\gamma H}{h}\right)^{\frac{1}{2}} \|v-P_{j,i}^m(v)\|_{V_{j}(K_{i})}
\leq  a_1^{\frac{1}{2}}\left( \Lambda_{j,i}+\cfrac{\gamma H}{h}\right)^{\frac{1}{2}} \|v\|_{V_{j}(K_{i})}.
\end{equation*}
Thus we have
\[
I\leq a_1^{\frac{1}{2}} (\lambda_{j,l_{j,i}^m+1}^{{(i)}})^{-\frac{1}{2}} \left( \Lambda_{j,i}+\cfrac{\gamma H}{h}\right)^{\frac{1}{2}}\|u_{H}^{m+1}-u_{H}^{m}\|_{a,K_{i}}.
\]
Using \eqref{err1}, we get
\[
S_{j,i}^{m+1}\leq\Big(\frac{\lambda_{j,l_{j,i}^m+1}^{{(i)}}}{\lambda_{j,l_{j,i}^{m+1}+1}^{{(i)}}}\Big)^{\frac{1}{2}}S_{j,i}^{m}
+a_1^{\frac{1}{2}}\left(\cfrac{\Lambda_{j,i}}{\lambda_{j,l_{j,i}^{m+1}+1}^{{(i)}}}+\cfrac{\gamma H}{h\lambda_{j,l_{j,i}^{m+1}+1}^{{(i)}}}\right)^{\frac{1}{2}}\|u_{\text{ms}}^{m+1}-u_{\text{ms}}^{m}\|_{a,K_{i}}.
\]
Hence, \eqref{eq:Sbound} is proved. \end{proof}

\subsection{Proof of Theorem \ref{thm:conv}}

In this section, we prove the convergence of the adaptive enrichment
algorithm. First of all, we recall that
\begin{equation}
\eta_{j,i}^{2}=\|R_{j,i}\|^{2}(\lambda_{j,l_{j,i}^m+1}^{{(i)}})^{-1} = (S^m_{j,i})^2.
\label{eq:etaS}
\end{equation}
We will use the single index notation $\eta_J$ and $S^m_J$
for $\eta_{j,i}$ and $S^m_{j,i}$ respectively.

Let $0<\theta<1$. We choose an index set $I$ so that
\begin{equation}
\theta^{2}\sum_{J=1}^{2N}\eta_{J}^{2}\leq\sum_{J\in I}\eta_{J}^{2}.\label{eq:indicator}
\end{equation}
We will then add basis function from $V^{i,\text{snap}}_j$ with
$J\in I$. Then, using Theorem \ref{thm:post} and \eqref{eq:indicator},
we have
\[
\theta^{2}\|u_h-u_{H}^{m}\|_{a}^{2}\leq\theta^{2}C_{\text{err}}\sum_{J=1}^{2N}\eta_{J}^{2}\leq C_{\text{err}}\sum_{J\in I}\eta_{J}^{2}.
\]
By (\ref{eq:etaS}), we also have
\begin{equation}
\|u_h-u_{H}^{m}\|_{a}^{2}\leq \cfrac{C_{\text{err}}}{\theta^{2}}\sum_{J\in I}(S_{J}^{m})^{2}. \label{eq:conv2}
\end{equation}

On the other hand,
\[
\sum_{J=1}^{N}(S_{J}^{m+1})^{2}=\sum_{J\in I}(S_{J}^{m+1})^{2}+\sum_{J\notin I}(S_{J}^{m+1})^{2}.
\]
By lemma \ref{lem:recurR}, we have
\begin{eqnarray*}
\sum_{J=1}^{N}(S_{J}^{m+1})^{2} & \leq & \sum_{J\in I}\Big((1+\alpha)\frac{\lambda_{j,l_{j,i}^m+1}^{{(i)}}}{\lambda_{j,l_{j,i}^{m+1}+1}^{{(i)}}}(S_{J}^{m})^{2}+(1+\alpha^{-1})a_1 D\|u_{H}^{m+1}-u_{H}^{m}\|_{a,K_{i}}^{2}\Big)\\
 &  & +\sum_{J\notin I}\left((1+\alpha)(S_{J}^{m})^{2}+(1+\alpha^{-1})a_1 D\|u_{H}^{m+1}-u_{H}^{m}\|_{a,K_{i}}^{2}\right).
\end{eqnarray*}
We assume the enrichment is obtained so that
\[
\delta :=\max_{J\in I} \frac{\lambda_{j,l_{j,i}^m+1}^{{(i)}}}{\lambda_{j,l_{j,i}^{m+1}+1}^{{(i)}}} \leq \delta_0 <1,
\]
where $\delta_0$ is independent of $m$.
We then have
\[
\sum_{J=1}^{2N}(S_{J}^{m+1})^{2}\leq(1+\alpha)\sum_{J=1}^{2N}(S_{J}^{m})^{2}-(1+\alpha)(1-\delta_0)\sum_{J\in I}(S_{J}^{m})^{2}+\delta L\|u_{H}^{m+1}-u_{H}^{m}\|_{a}^{2},
\]
where
\begin{equation}
L_{m+1}=N_E(1+\alpha^{-1}) a_1 \Big( \max_{1\leq i\leq N} \max_{1\leq j\leq 2} D \Big), \label{eq:assumeL2}
\end{equation}
where $N_E$ is the maximum number of edges of coarse grid blocks,
and we also emphasise that $L_{m+1}$ depends on $m$.
By (\ref{eq:indicator}),
\[
\sum_{J=1}^{2N}(S_{J}^{m+1})^{2}\leq(1+\alpha)\sum_{J=1}^{2N}(S_{J}^{m})^{2}-(1+\alpha)(1-\delta_0)\theta^{2}\sum_{J=1}^{2N} (S_{J}^{m})^{2}+\delta L_{m+1}\|u_{H}^{m+1}-u_{H}^{m}\|_{a}^{2}.
\]
Let $\rho=(1+\alpha)(1-(1-\delta_0)\theta^{2})$. We choose $\alpha>0$
small enough so that $0<\rho<1$. The above is then written as
\begin{equation}
\sum_{J=1}^{2N}(S_{J}^{m+1})^{2}\leq\rho\sum_{J=1}^{2N}(S_{J}^{m})^{2}+\delta L_{m+1} \|u_{H}^{m+1}-u_{H}^{m}\|_{a}^{2}.
\end{equation}

Note that, by Galerkin orthogonality, we have
\[
\|u_{H}^{m+1}-u_{H}^{m}\|_{a}^{2}=\|u_h-u_{H}^{m}\|_{a}^{2}-\|u_h-u_{H}^{m+1}\|_{a}^{2}.
\]
So, we have
\begin{equation}
\sum_{J=1}^{2N}(S_{J}^{m+1})^{2}\leq\rho\sum_{J=1}^{2N}(S_{J}^{m})^{2}+\delta L_{m+1} (\|u_h-u_{H}^{m}\|_{a}^{2}-\|u_h-u_{H}^{m+1}\|_{a}^{2})
\end{equation}
which implies
\begin{equation}
\|u_h-u_{H}^{m+1}\|_{a}^{2} +
\cfrac{1}{\delta L_{m+1}}\sum_{J=1}^{2N}(S_{J}^{m+1})^{2}\leq\|u_h-u_{H}^{m}\|_{a}^{2}+\frac{\rho}{\delta L_{m+1}}\sum_{J=1}^{2N}(S_{J}^{m})^{2}.\label{eq:conv1}
\end{equation}

Finally, using \eqref{eq:conv2},
\begin{equation*}
 \|u_h-u_{H}^{m+1}\|_{a}^{2}+\frac{1}{\delta L_{m+1}}\sum_{J=1}^{2N}(S_{J}^{m})^{2}\leq(1-\beta)\|u_h-u_{H}^{m}\|_{a}^{2}+(\frac{\beta C_{err}}{\theta^2}+\frac{\rho}{\delta L_{m+1}})\sum_{J=1}^{2N}(S_{J}^{m})^{2}.
\end{equation*}
 Let $\beta=\cfrac{\theta^2(1-\rho L_m/L_{m+1})}{\theta^2 + C_{\text{err}}\delta L_{m}}$ and combining the above with \eqref{eq:conv1}, we obtain
\begin{equation} \|u-u_{\text{h}}^{m+1}\|_{a}^{2}+\cfrac{1}{\delta L_{m+1}}\sum_{J=1}^{2N}(S_{J}^{m+1})^{2} \leq(1-\beta)\|u-u_{\text{h}}^{m}\|_{a}^{2}+\frac{(1-\beta)}{\delta L_{m}}\sum_{J=1}^{2N}(S_{J}^{m})^{2}.
\end{equation}
Hence, Theorem \ref{thm:conv} is proved.

 \bibliographystyle{plain}
\bibliography{references}

\begin{thebibliography}{10}

\bibitem{abdul_yun}
Assyr Abdulle and Yun Bai.
\newblock Adaptive reduced basis finite element heterogeneous multiscale
  method.
\newblock {\em Comput. Methods Appl. Mech. Engrg.}, 257:203--220, 2013.

\bibitem{Arbogast_two_scale_04}
T.~Arbogast.
\newblock Analysis of a two-scale, locally conservative subgrid upscaling for
  elliptic problems.
\newblock {\em SIAM J. Numer. Anal.}, 42(2):576--598 (electronic), 2004.

\bibitem{BrennerScott}
S.~Brenner and L.~Scott.
\newblock {\em The Mathematical Theory of Finite Element Methods}.
\newblock Springer-Verlag, New York, 2007.

\bibitem{donoho_sparsity_review}
Scott~Shaobing Chen, David~L. Donoho, and Michael~A. Saunders.
\newblock Atomic decomposition by basis pursuit.
\newblock {\em SIAM Rev.}, 43(1):129--159, 2001.
\newblock Reprinted from SIAM J. Sci. Comput. {{\bf{2}}0} (1998), no. 1, 33--61
  (electronic) [ MR1639094 (99h:94013)].

\bibitem{Chu_Hou_MathComp_10}
C.-C. Chu, I.~G. Graham, and T.-Y. Hou.
\newblock A new multiscale finite element method for high-contrast elliptic
  interface problems.
\newblock {\em Math. Comp.}, 79(272):1915--1955, 2010.

\bibitem{ReducedCon}
E.~Chung and Y.~Efendiev.
\newblock Reduced-contrast approximations for high-contrast multiscale flow
  problems.
\newblock {\em Multiscale Model. Simul.}, 8:1128--1153, 2010.

\bibitem{Wave}
E.~Chung, Y.~Efendiev, and R.~Gibson.
\newblock An energy-conserving discontinuous multiscale finite element method
  for the wave equation in heterogeneous media.
\newblock {\em Advances in Adaptive Data Analysis}, 3:251--268, 2011.

\bibitem{WaveGMsFEM}
E.~Chung, Y.~Efendiev, and W.~T. Leung.
\newblock Generalized multiscale finite element method for wave propagation in
  heterogeneous media.
\newblock {\em arXiv:1307.0123}.

\bibitem{Adaptive-GMsFEM}
E.~Chung, Y.~Efendiev, and G.~Li.
\newblock An adaptive {GM}s{FEM} for high contrast flow problems.
\newblock {\em J. Comput. Phys.}, 273:54--76, 2014.

\bibitem{MsDG}
E.~Chung and W.~T. Leung.
\newblock A sub-grid structure enhanced discontinuous galerkin method for
  multiscale diffusion and convection-diffusion problems.
\newblock {\em Commun. Comput. Phys.}, 14:370--392, 2013.

\bibitem{Dorfler96}
W.~Dorfler.
\newblock A convergent adaptive algorithm for poisson's equation.
\newblock {\em SIAM J.Numer. Anal.}, 33:1106 -- 1124, 1996.

\bibitem{ohl12}
Martin Drohmann, Bernard Haasdonk, and Mario Ohlberger.
\newblock Reduced basis approximation for nonlinear parametrized evolution
  equations based on empirical operator interpolation.
\newblock {\em SIAM J. Sci. Comput.}, 34(2):A937--A969, 2012.

\bibitem{dur91}
L.J. Durlofsky.
\newblock Numerical calculation of equivalent grid block permeability tensors
  for heterogeneous porous media.
\newblock {\em Water Resour. Res.}, 27:699--708, 1991.

\bibitem{ee03}
W.~E and B.~Engquist.
\newblock Heterogeneous multiscale methods.
\newblock {\em Comm. Math. Sci.}, 1(1):87--132, 2003.

\bibitem{egh12}
Y.~Efendiev, J.~Galvis, and T.~Hou.
\newblock Generalized multiscale finite element methods.
\newblock {\em Journal of Computational Physics}, 251:116--135, 2013.

\bibitem{eglmsMSDG}
Y.~Efendiev, J.~Galvis, R.~Lazarov, M.~Moon, and M.~Sarkis.
\newblock Generalized multiscale finite element method. {S}ymmetric interior
  penalty coupling.
\newblock {\em J. Comput. Phys.}, 255:1--15, 2013.

\bibitem{eglp13}
Y.~Efendiev, J.~Galvis, G.~Li, and M.~Presho.
\newblock Generalized multiscale finite element methods. oversampling
  strategies.
\newblock to appear in International Journal for Multiscale Computational
  Engineering.

\bibitem{egw10}
Y.~Efendiev, J.~Galvis, and X.H. Wu.
\newblock Multiscale finite element methods for high-contrast problems using
  local spectral basis functions.
\newblock {\em Journal of Computational Physics}, 230:937--955, 2011.

\bibitem{eh09}
Y.~Efendiev and T.~Hou.
\newblock {\em {Multiscale Finite Element Methods: Theory and Applications}},
  volume~4 of {\em Surveys and Tutorials in the Applied Mathematical Sciences}.
\newblock Springer, New York, 2009.

\bibitem{ehg04}
Y.~Efendiev, T.~Hou, and V.~Ginting.
\newblock Multiscale finite element methods for nonlinear problems and their
  applications.
\newblock {\em Comm. Math. Sci.}, 2:553--589, 2004.

\bibitem{GhommemJCP2013}
M.~Ghommem, M.~Presho, V.~M. Calo, and Y.~Efendiev.
\newblock Mode decomposition methods for flows in high-contrast porous media.
  global–-local approach.
\newblock {\em Journal of Computational Physics, Vol. 253.}, pages 226–--238.

\bibitem{hw97}
T.~Hou and X.H. Wu.
\newblock A multiscale finite element method for elliptic problems in composite
  materials and porous media.
\newblock {\em J. Comput. Phys.}, 134:169--189, 1997.

\bibitem{dinh13}
Dinh Bao~Phuong Huynh, David~J. Knezevic, and Anthony~T. Patera.
\newblock A static condensation reduced basis element method: approximation and
  {\it a posteriori} error estimation.
\newblock {\em ESAIM Math. Model. Numer. Anal.}, 47(1):213--251, 2013.

\bibitem{AdaptiveFEM}
K.~Mekchay and R.~H. Nochetto.
\newblock Convergence of adaptive finite element method for general second
  order elliptic {PDE}s.
\newblock {\em SIAM J. Numer. Anal.}, 43:1803--1827, 2005.

\bibitem{nguyen13}
N.~C. Nguyen, G.~Rozza, D.~B.~P. Huynh, and A.~T. Patera.
\newblock Reduced basis approximation and a posteriori error estimation for
  parametrized parabolic {PDE}s: application to real-time {B}ayesian parameter
  estimation.
\newblock In {\em Large-scale inverse problems and quantification of
  uncertainty}, Wiley Ser. Comput. Stat., pages 151--177. Wiley, Chichester,
  2011.

\bibitem{IPDGbook}
Beatrice~M. Riviere.
\newblock {\em Discontinuous {G}alerkin Methods For Solving Elliptic And
  parabolic Equations: Theory and Implementation}.
\newblock SIAM, 2008.

\bibitem{tonn11}
Timo Tonn, K.~Urban, and S.~Volkwein.
\newblock Comparison of the reduced-basis and {POD} {\it a posteriori} error
  estimators for an elliptic linear-quadratic optimal control problem.
\newblock {\em Math. Comput. Model. Dyn. Syst.}, 17(4):355--369, 2011.

\bibitem{weh02}
X.H. Wu, Y.~Efendiev, and T.Y. Hou.
\newblock Analysis of upscaling absolute permeability.
\newblock {\em Discrete and Continuous Dynamical Systems, Series B.},
  2:158--204, 2002.

\end{thebibliography}
\end{document}